\documentclass{article}
\usepackage{graphicx} % Required for inserting images
\usepackage{longtable}
\usepackage{booktabs} 
\usepackage{amsmath, amssymb, amsfonts, amsthm, mathrsfs}

\usepackage{graphicx}
\usepackage{subcaption}
\usepackage{booktabs}
\usepackage{float}
\usepackage{afterpage}
\usepackage[hidelinks]{hyperref}
\usepackage{bookmark}
\usepackage{xcolor}
\usepackage[capitalize]{cleveref}
\usepackage{cases}
\usepackage{enumitem}
\usepackage{geometry} 
\usepackage{natbib}

\newtheorem{theorem}{Theorem}[section]
\newtheorem{lemma}[theorem]{Lemma}
\newtheorem{corollary}[theorem]{Corollary}
\newtheorem{definition}[theorem]{Definition}

\theoremstyle{definition}  % 设置为正常字体（非斜体）
\newtheorem{example}{Example}[section]
\newtheorem{remark}{Remark}
\theoremstyle{plain}  % 粗体标题+斜体内容（常用于命题、定理、假设等）
\newtheorem*{assumption}{Assumption}

\title{Onsager--Machlup Functionals for Generalized Newtonian Equations of Motion with Time-Varying Fractional Noise}

\author{
Yanbin Zhu%
\thanks{This work is supported by the National Key R\&D Program of China (No.~2023YFA1009200), the National Key Project of the National Natural Science Foundation of China (Grant No.~12531009), and the National Natural Science Foundation of China (Grant Nos.~12471183, 12071175).}
\and
Xiaomeng Jiang%
\thanks{Corresponding author.}
\and
Yong Li
}

\date{}

\begin{document}
\maketitle

\begin{center}
College of Mathematics, Jilin University, Changchun 130012, PR China
\end{center}

\begin{center}
\texttt{zhuyb23@mails.jlu.edu.cn},
\texttt{jxmlucy@hotmail.com},
\texttt{liyong@jlu.edu.cn}
\end{center}

\begin{abstract}
In this paper, we derive the Onsager--Machlup functional for a class of
degenerate stochastic differential equations on $\mathbb{R}^{m+n}$ driven by
$n$-dimensional fractional Brownian motion with time-dependent diffusion
coefficients, where the Hurst parameter satisfies $H\in(1/4,1)$.
The main difficulty arises from the interaction between the degenerate
structure and the multidimensional fractional noise, which produces nontrivial
coupling terms under small-ball conditioning. By combining the Gaussian
correlation inequality with an approximation argument for infinite-dimensional
convex sets, we establish a decoupling mechanism that allows these coupling
effects to be controlled by unconditional Gaussian expectations.
Furthermore, through regularity estimates for the degenerate components and
sharp H\"older norm estimates, we eliminate the remaining coupling
contributions and obtain an explicit expression for the Onsager--Machlup
functional. As applications of the derived functional, we obtain the
corresponding constrained Euler--Lagrange equations characterizing the most
probable transition paths of the system, and establish a sufficient condition
under which the stochastic differential equation preserves the most probable
path.
\end{abstract}

\medskip
\noindent\textbf{Keywords.}
Fractional Brownian motion; Onsager--Machlup functional;  Constrained Euler--Lagrange equations; Most probable paths; Small-ball probability.

\medskip
\noindent\textbf{MSC (2020).}
60H10; 60G22; 49K15; 60F10.

\allowdisplaybreaks

\section{Introduction}
In classical deterministic dynamics, trajectories trapped in a potential well typically converge to a stable equilibrium. In realistic environments, however, random fluctuations are unavoidable. The presence of noise gives rise to metastability: trajectories may remain near the deterministic equilibrium for long periods of time before making rare transitions to other regions of the state space.

Although the system may remain near a given potential well for a long period, the cumulative effect of noise eventually induces it to cross energy barriers, resulting in state-switching events \cite{kramers1940brownian}. While such transitions appear irregular, they exhibit a clear probabilistic structure in path space: within a vanishingly small tube around a reference trajectory, the relative probability density of the stochastic paths concentrates along a distinguished configuration that minimizes the transition governing functional. 

The Onsager--Machlup (OM) functional is a fundamental mathematical tool for describing this phenomenon. The OM functional was first introduced by Onsager and Machlup \cite{Onsager1953} and rigorously established for SDEs under the supremum norm in \cite{ikeda1981stochastic}. Analogous to the Lagrangian in classical mechanics, the OM functional identifies the most probable transition path by acting as an action functional on path space. This framework is of central importance in chemical reaction kinetics, protein folding, and neuronal dynamics \cite{BATTEZZATI2013163, 10.1063/1.467139, faccioli2006dominant}.

The paper extends the classical Onsager--Machlup theory in two important directions:
\begin{enumerate}
\renewcommand{\labelenumi}{(\roman{enumi})}

\item \textbf{Time-Varying Fractional Noise.}
In complex media such as viscoelastic fluids and protein structures,
noise often exhibits memory effects and temporal correlations
\cite{kou2008stochastic}.
By introducing fractional Brownian motion with time-dependent diffusion
coefficients, one can model systems in which both the noise intensity and
the temporal correlations evolve over time.
These memory effects imply that the most probable transition path is no
longer determined solely by the instantaneous state of the system, but is
instead coupled to the entire history of the trajectory.

\item \textbf{Degenerate Structures in Physical Systems.}
A representative class of examples arises from stochastic Newtonian
equations of motion, or more generally from stochastic Hamiltonian systems
\cite{talay2002stochastic}.
In these systems, random perturbations typically act as fluctuating forces
that directly modify the velocity (or momentum) variables rather than the
position variables.
By denoting $X_t$ and $Y_t$ as the position and velocity of the system,
respectively, the dynamics can be written as
\[
\begin{cases}
X_t' = Y_t,\\[1mm]
Y_t' = F_t(X_t,Y_t)+\sigma_t\xi_t^{H},
\end{cases}
\]
where $F_t$ denotes the deterministic force field,
$f'(t)$ denotes the derivative of $f(t)$,
and $\xi_t^{H}$ represents fractional  white noise.
In such degenerate settings, stochastic forcing is injected into only a
subset of the degrees of freedom, namely the velocity component, and is
then propagated throughout the entire system via the intrinsic kinematic
coupling.
This classical physical structure provides a natural prototype for the
more general class of coupled nonlinear degenerate dynamics investigated
in the present work.

\end{enumerate}

In this paper, we investigate a degenerate system driven by time-dependent fractional noise:
\begin{equation}
\begin{cases}
dX_t = a_t(X_t,Y_t) \, dt, \\[4pt]
dY_t = b_t(X_t, Y_t) \, dt + \sigma_t \, dB^H_t,
\end{cases} \label{fir}
\end{equation}
with initial condition $(X_0,Y_0)=(x_0,y_0)$. Here, $a_t(x,y):[0,1]\times \mathbb{R}^m\times\mathbb{R}^n\to \mathbb{R}^m$ and $b_t(x,y):[0,1]\times \mathbb{R}^m\times\mathbb{R}^n \to \mathbb{R}^n$ are the drift coefficients, $\sigma_t = \text{diag}(\sigma^1_t, \dots, \sigma^n_t)$ is the diffusion matrix, and $B^H$ is an $n$-dimensional fractional Brownian motion with Hurst parameter $H \in (1/4, 1)$. For a given reference path $\phi -y_0\in \mathcal{H}^\sigma_H([0,1]; \mathbb{R}^n)$ representing the trajectory of the component $Y$, we define its response state path $\psi \in C^{1+H}([0,1]; \mathbb{R}^m)$ as the unique solution to the controlled deterministic system:
\begin{equation}
\begin{cases}
d\psi_t = a_t(\psi_t, \phi_t) \, dt, \\[4pt]
\psi_0 = x_0.
\end{cases} \label{sec}
\end{equation}
The OM functional characterizes the asymptotic probability of the stochastic process $Y$ remaining within a small $\beta$-H\"{o}lder neighborhood of the reference path $\phi$. Specifically, if the following limit exists for a suitable H\"{o}lder exponent $\max\{H-1/2,0\} < \beta < H-1/4$:
\begin{equation}
\exp\bigl(-J(\phi)\bigr) = \lim_{\varepsilon \to 0} 
\frac{\mathbb{P}\bigl(\| Y_{\cdot} - \phi_{\cdot} \|_{\beta} \leq \varepsilon\bigr)}
{\mathbb{P}\bigl(\Vert \int_0^\cdot \sigma_s \, dB^{H}_s  \Vert_\beta \leq \varepsilon \bigr)}, \label{omf}
\end{equation}
then $J(\phi)$ is called the OM functional associated with the system \eqref{fir} around the path $\phi$.

 The derivation of OM functionals for stochastic differential equations driven by fractional Brownian motion was initiated in \cite{Nualart} through a fractional Girsanov transformation. Since then, the theory has undergone substantial development in several directions. In particular, \cite{zhang2025persistenceinvarianttoristochastic} combined OM theory with KAM techniques to investigate stochastic nonlinear Schrödinger equations, while \cite{duanjq} introduced a probability-flow framework for jump-diffusion processes. Furthermore, \cite{doi:10.1137/20M1310539} established the $\Gamma$-convergence of the OM functional to the geometric Freidlin--Wentzell action functional. 
Extensions to stochastic systems with time-dependent coefficients were established in \cite{Zhang2026}, while degenerate stochastic systems driven by standard Brownian motion were studied in \cite{liu2024onsager}.

To the best of our knowledge, the OM functional for degenerate stochastic differential equations driven by multidimensional time-varying fractional Brownian motion has not been established in the existing literature.  The combination of multidimensionality, time-dependent coefficients, and degeneracy in the fractional setting introduces several substantial analytical difficulties. A first challenge arises from implementing a degenerate Girsanov transformation on the path space. Moreover, when performing a Taylor expansion of the non-degenerate drift coefficient $b_t(X_t,Y_t)$ around a reference path, two distinct coupling mechanisms emerge and complicate the asymptotic analysis:

\begin{itemize}
    \item \textbf{Degeneracy-induced coupling.}
    The dependence of $b_t$ on the degenerate component $X_t$ propagates through the Taylor expansion. Since the dynamics of $X_t$ are determined implicitly through the coupling relation $a_t(X_t,Y_t)$, this induces a non-local dependence between the degenerate and non-degenerate subsystems.

    \item \textbf{Multidimensional coupling.}
    The multidimensional structure of the fractional Brownian motion generates non-diagonal cross-interaction terms in the exponential functional. Under the Hölder small-ball conditioning $|\cdot|_\beta\le \varepsilon$, these terms remain strongly coupled and prevent a direct asymptotic evaluation.
\end{itemize}

	The main obstruction is that the coupling terms generated by the degenerate structure and the multidimensional fractional noise fall outside the class of exponential functionals that can be treated by existing fractional Onsager--Machlup techniques.

Indeed, in the classical fractional setting, the asymptotic analysis of the exponential terms typically relies on three mechanisms. The first concerns stochastic integrals against fractional Brownian motion with deterministic square-integrable integrands.  The second involves double stochastic integrals that can be reduced, via the stochastic Fubini theorem, to stochastic integrals with deterministic kernels. The remaining higher-order terms are sufficiently small and can be controlled through exponential martingale estimates, yielding contributions of order $o(1)$ under the small-ball conditioning.

However, the coupling terms arising in the present degenerate multidimensional setting do not belong to any of these categories. The degeneracy-induced interaction generates random integrands depending implicitly on the coupled state variables, while the multidimensional structure produces non-diagonal cross-interaction terms that cannot be reduced to deterministic-kernel representations. Consequently, the existing fractional OM framework no longer applies, and new techniques are required to control these conditional exponential functionals.

As a consequence, the conditional expectations arising in the OM analysis cannot be handled by existing approaches developed for fractional-noise-driven systems.

To overcome these difficulties, our analysis relies on three main ingredients:

\begin{itemize}
    \item \textbf{Degeneracy reduction through operator selection.}
    We perform a degenerate Girsanov transformation by shifting only the component directly driven by the noise, namely the non-degenerate component $Y$. The trajectory of the degenerate component $X$ is then reconstructed through the deterministic coupling relation \eqref{sec} via a Volterra-type integral operator $K_H^\sigma$.

    \item \textbf{Conditional expectation estimates via the Gaussian Correlation Inequality.}
    To control the multidimensional cross-interaction terms, we extend the Gaussian correlation inequality (GCI) to the abstract Wiener space setting. We prove that the conditional expectation under the Hölder small-ball constraint is bounded by its unconditional counterpart (Theorem~\ref{conditional-expectation}), thereby reducing the asymptotic analysis to an uncoupled form.

    \item \textbf{Fractional regularity and Gronwall estimates.}
   Combining a refined regularity analysis of fractional operators with suitable Gronwall-type inequalities, we establish sharp $(1+\beta)$-Hölder norm estimates for the degenerate components. By combining these estimates with suitable stochastic integration-by-parts formulas, we obtain effective control of the remaining degeneracy-induced coupling terms and show that they vanish asymptotically.
\end{itemize}

Our main result (Theorem~\ref{result}) provides explicit expressions for the Onsager--Machlup (OM) functional in the three regimes $H\in(1/4,1)$:
\begin{equation}\label{mainresultfunctional}
J(\phi)=
\begin{cases}
\dfrac{1}{2}\displaystyle\int_0^1
\Big|
\dot{\phi}_s
-
s^{\alpha}D_{0^+}^{\alpha}s^{-\alpha}
\sigma_s^{-1}b_s(\psi_s,\phi_s)
\Big|^2
+d_H\nabla_y\!\cdot b_s(\psi_s,\phi_s)\,ds,
& \tfrac{1}{2}<H<1, \\[1.5em]

\dfrac{1}{2}\displaystyle\int_0^1
\left|
\sigma_s^{-1}
\bigl(\phi'_s-b_s(\psi_s,\phi_s)\bigr)
\right|^2
+\nabla_y\!\cdot b_s(\psi_s,\phi_s)\,ds,
& H=\tfrac{1}{2}, \\[1.5em]

\dfrac{1}{2}\displaystyle\int_0^1
\Big|
\dot{\phi}_s
-
s^{-\alpha}I_{0^+}^{\alpha}s^{\alpha}
\sigma_s^{-1}b_s(\psi_s,\phi_s)
\Big|^2
+d_H\nabla_y\!\cdot b_s(\psi_s,\phi_s)\,ds,
& \tfrac{1}{4}<H<\tfrac{1}{2},
\end{cases}
\end{equation}
where $\alpha=|H-\tfrac12|$, $d_H$ is a constant depending only on $H$, and the degenerate component $\psi$ is implicitly coupled with $\phi$ through the system \eqref{sec}. For $H\neq \tfrac12$, the functions $\phi$ and $\dot{\phi}$ are related through
\[
\phi_t-y_0=K_H^\sigma(\dot{\phi})(t),
\]
where $K_H^\sigma$ is a fractional integral operator of order $H+\tfrac12$, whose precise definition will be given in Section~2. For the Brownian case $H=\tfrac12$, $\phi'_t$ denotes the ordinary time derivative of $\phi_t$.

As an application of the derived Onsager--Machlup functional, we study both
most probable transition paths and most probable paths associated with the SDE
\eqref{fir}. A most probable transition path is defined as a minimizer of
$J(\cdot)$ among all admissible paths satisfying the boundary conditions
$\phi_0=y_0$ and $\phi_1=y_1$. By contrast, a most probable path is considered
under the prescribed initial condition only, without imposing a terminal
constraint at $t=1$.

Furthermore, by incorporating the non-local differential constraint through a
time-dependent Lagrange multiplier, we derive the corresponding constrained
Euler--Lagrange system. This system provides a deterministic characterization
of the most probable transition paths associated with the underlying degenerate
fractional stochastic dynamics. In addition, we establish a sufficient condition
under which the SDE preserves the most probable path, thereby extending the
result of \cite{xinze2026most}.

The remainder of this paper is organized as follows. In Section 2, we introduce the mathematical framework and collect several preliminary results that will be used throughout the paper, including aspects of fractional calculus, infinite-dimensional Gaussian measures, and a proof of the GCI.
Section 3 is devoted to the proof of the main result, Theorem~\ref{result}. The proof combines a degenerate Girsanov transformation with delicate conditional expectation estimates under small-ball constraints.
In Section 4, we derive the Euler–Lagrange equations satisfied by the most probable transition path between two points for an SDE via the OM functional. Furthermore, we establish the preservation theorem of the most probable path under free terminal conditions, followed by several numerical experiments to illustrate the aforementioned theoretical results.

\section{Preliminaries}
In this section, we recall the foundational definitions, notations, and preliminary lemmas concerning H\"{o}lder spaces and fractional calculus that will be utilized throughout this paper.

Throughout this paper, $|\cdot|$ denotes the Euclidean norm of a vector. For a measurable function $f$ defined on the underlying probability space $(\Omega,\mathcal{F},\mathbb{P})$, we write
\[
\mathbb{E}[f]
=
\int_\Omega f(\omega)\,d\mathbb{P}(\omega)
\]
for its expectation. The notation $\mathbb{E}(\,\cdot\,|A)$ denotes the conditional expectation given an event $A$, while
\[
\mathbb{E}(\,\cdot\,|A,B)
:=
\mathbb{E}(\,\cdot\,|A\cap B)
\]
denotes the conditional expectation with respect to the intersection of the events $A$ and $B$.
 
  For two vectors $\alpha,\beta\in\mathbb{R}^n$, their Euclidean inner product is denoted by
\[
\alpha\cdot\beta=\sum_{i=1}^n \alpha_i\beta_i.
\]
Moreover, for an $n$-dimensional stochastic integral, we use the notation
\[
\int_0^T \phi_s\cdot dW_s
:=\sum_{i=1}^n\int_0^T \phi_s^i\,dW_s^i.
\]

For a vector-valued function $b(x)$, we use $\nabla_x b(x)$ to denote its Jacobian matrix with respect to $x$, and $\nabla_x\cdot b(x)$ to denote its divergence.

\subsection{Function Spaces and Norms}
We first introduce the standard H\"{o}lder spaces on the time interval $[0,1]$. Let $\beta \in (0, 1]$. We denote by $C^\beta([0,1]; \mathbb{R})$ the Banach space of $\beta$-H\"{o}lder continuous scalar functions equipped with the classical semi-norm
\begin{equation}\label{eq:Holder-seminorm}
    [f]_\beta := \sup_{0 \le s < t \le 1} \frac{|f(t) - f(s)|}{|t - s|^\beta} < \infty.
\end{equation}
To facilitate the component-wise analysis of our multi-dimensional system, for any vector-valued function $f = (f^1, \dots, f^d)^T \in C^\beta([0,1]; \mathbb{R}^d)$, we define its H\"{o}lder semi-norm under the component-wise supremum topology as
\begin{equation*}
    [f]_\beta := \max_{1 \le i \le d} [f^i]_\beta.
\end{equation*}
The space $C^\beta([0,1]; \mathbb{R}^d)$ is standardly endowed with the full norm
\begin{equation*}
    \|f\|_{C^\beta([0,1]; \mathbb{R}^d)} := \max_{1 \le i \le d} \sup_{t \in [0,1]} |f^i(t)| + [f]_\beta.
\end{equation*}
For higher regularity where $1 < 1+\beta < 2$, the space $C^{1+\beta}([0,1]; \mathbb{R}^d)$ consists of continuously differentiable functions whose derivatives are $\beta$-H\"{o}lder continuous. Its standard norm combines the supremum of both the function and its derivative:
\begin{equation*}
    \|f\|_{C^{1+\beta}([0,1]; \mathbb{R}^d)} := \max_{1 \le i \le d} \sup_{t \in [0,1]} |f^i(t)| + \max_{1 \le i \le d} \sup_{t \in [0,1]} |(f^i)'(t)| + [f']_{\beta}.
\end{equation*}

In this paper, the stochastic processes and reference paths under consideration fundamentally vanish at the origin (or we consider their deviations from the initial state $x_0$). This structural property allows for a significant simplification of the geometric topologies.

Let $C^\beta_0([0,1]; \mathbb{R}^d) := \{ f \in C^\beta([0,1]; \mathbb{R}^d) : f(0) = 0 \}$ be the closed subspace of functions vanishing at the origin. For any \( f \in C^\beta_0([0,1]; \mathbb{R}^d) \), the inequality
\[
\sup_{t \in [0,1]} |f^i(t)| \le [f^i]_\beta
\]
holds trivially. Thus, when the initial value is zero, the semi-norm \([f]_\beta\) itself is already a proper norm, equivalent to the full H\"{o}lder norm.

Similarly, we define the higher regularity subspace pinned at the origin as $C^{1+\beta}_0([0,1]; \mathbb{R}^d) := \{ f \in C^{1+\beta}([0,1]; \mathbb{R}^d) : f(0) = 0 \}$. By virtue of the fundamental theorem of calculus, $f^i(t) = \int_0^t (f^i)'(s) \, ds$, which implies $\sup_{t \in [0,1]} |f^i(t)| \le \sup_{t \in [0,1]} |(f^i)'(t)|$. Hence, the position supremum becomes redundant and can be omitted.

Throughout the remainder of this paper, since all relevant trajectories and error terms are analyzed within the zero-initial-condition subspaces $C^\beta_0$ and $C^{1+\beta}_0$, we adopt the strict convention that the notations $\|\cdot\|_\beta$ and $\|\cdot\|_{1+\beta}$ refer exclusively to the streamlined norms:
\begin{align*}
    \|f\|_\beta &:= [f]_\beta = \max_{1 \le i \le d} \sup_{0 \le s < t \le 1} \frac{|f^i(t) - f^i(s)|}{|t - s|^\beta}, \\
    \|f\|_{1+\beta} &:= \max_{1 \le i \le d} \sup_{t \in [0,1]} |(f^i)'(t)| + [f']_{\beta}.
\end{align*}
Under this convention, for any
$f\in C^{1+\beta}_0([0,1];\mathbb{R}^d)$, we have
\[
    \|f\|_{1+\beta}
    =
    \|f'\|_{C^\beta([0,1];\mathbb{R}^d)} .
\]

\subsection{Fractional Calculus}
We present the core machinery of fractional calculus required for handling fractional Brownian motion. For a comprehensive treatise, we refer the reader to Samko et al.~\cite{Samko1993}.

\begin{definition}
    Let $f \in L^1([a,b];\mathbb{R})$. For $\alpha > 0$, the operators
    \begin{align*}
        (I_{a^+}^\alpha f)(x) &:= \frac{1}{\Gamma(\alpha)} \int_a^x (x - y)^{\alpha-1} f(y) \, dy, \quad x \ge a, \\
        (I_{b^-}^\alpha f)(x) &:= \frac{1}{\Gamma(\alpha)} \int_x^b (y - x)^{\alpha-1} f(y) \, dy, \quad x \le b,
    \end{align*}
    are called the left-sided and right-sided Riemann--Liouville fractional integrals of order $\alpha$, respectively.
\end{definition}

For any $\alpha > 0$, if $f \in L^p([a,b];\mathbb{R})$ and $g \in L^q([a,b];\mathbb{R})$ with $1/p + 1/q \le 1 + \alpha$ (and $p, q \ge 1$), the following fractional integration-by-parts formula holds:
\begin{equation}
    \int_a^b f(s) (I_{a^+}^\alpha g)(s) \, ds = \int_a^b (I_{b^-}^\alpha f)(s) g(s) \, ds. \label{ffubini}
\end{equation}
We denote by $I_{a^+}^\alpha(L^p)$ and $I_{b^-}^\alpha(L^p)$ the images of $L^p([a,b];\mathbb{R})$ under the fractional integral operators $I_{a^+}^\alpha$ and $I_{b^-}^\alpha$, respectively. 

\begin{definition}
    Let $f \in I_{a^+}^\alpha(L^p)$ and $g \in I_{b^-}^\alpha(L^p)$ with $p \ge 1$. The left-sided and right-sided Riemann--Liouville fractional derivatives of order $\alpha > 0$ are defined respectively by
    \begin{align*}
        (D_{a^+}^\alpha f)(x) &:= \left( \frac{d}{dx} \right)^{[\alpha] + 1} I_{a^+}^{1 + [\alpha] - \alpha} f(x), \\
        (D_{b^-}^\alpha g)(x) &:= \left( -\frac{d}{dx} \right)^{[\alpha] + 1} I_{b^-}^{1 + [\alpha] - \alpha} g(x),
    \end{align*}
    where $[\alpha]$ denotes the integer part of $\alpha$.
\end{definition}

From \eqref{ffubini}, we deduce that for $0 < \alpha < 1$, the dual relation for fractional derivatives reads
\begin{equation}
    \int_a^b f(s) (D_{a^+}^\alpha g)(s) \, ds = \int_a^b (D_{b^-}^\alpha f)(s) g(s) \, ds, \label{fffubini}
\end{equation}
provided that $f \in I^\alpha_{b^-}(L^p)$ and $g \in I^\alpha_{a^+}(L^q)$ with $1/p + 1/q \le 1 + \alpha$.

For any $f \in I_{a^+}^\alpha(L^p)$, the representation $f = I_{a^+}^\alpha(\phi)$ determines the function $\phi \in L^p$ uniquely. In this sense, the fractional derivative operator acts as the left inverse of the fractional integral operator.

When $\alpha p > 1$, any function in $I_{a^+}^\alpha(L^p)$ embeds continuously into the H\"{o}lder space $C^{\alpha - 1/p}([a,b])$. Conversely, every H\"{o}lder continuous function of order $\beta > \alpha$ admits a fractional derivative of order $\alpha$; see Decreusefond and \"Ust\"{u}nel~\cite[Proposition~2.1]{Decreusefond1999}. 

Crucially, for $0 < \alpha < 1$, these fractional derivatives can be evaluated explicitly via Marchaud's formula (or Weyl's representation, cf.~\cite[Remark~5.3]{Samko1993}). The left-sided derivative admits the representation
\begin{equation}
    D^\alpha_{a^+} f(x) = \frac{1}{\Gamma(1-\alpha)} \left( \frac{f(x)}{(x-a)^\alpha} + \alpha \int_a^x \frac{f(x) - f(y)}{(x-y)^{\alpha+1}} \, dy \right), \label{Weyl}
\end{equation}
where the improper integral converges in the $L^p$ sense. Analogously, the right-sided fractional derivative of order $\alpha$ is given by
\begin{equation}
    D_{T^-}^\alpha f(s) = \frac{1}{\Gamma(1-\alpha)} \left( \frac{f(s)}{(T-s)^\alpha} - \alpha \int_s^T \frac{f(u) - f(s)}{(u-s)^{\alpha+1}} \, du \right).
\end{equation}

		For vector-valued functions, fractional integration or differentiation is defined component-wise, by applying the respective fractional operators to each individual component.

 \subsection{The \texorpdfstring{$\sigma$}{sigma}-Weighted Fractional Cameron--Martin Space}
		
Let \(T>0\) and let
\[
    \Omega_H
    =
    C_0([0,T];\mathbb R^n)
\]
be the canonical path space, equipped with its Borel \(\sigma\)-algebra
\(\mathcal F_H\). For each \(H\in(0,1)\), we denote by \(\mathbb P_H\) the
Gaussian measure on \((\Omega_H,\mathcal F_H)\) under which the canonical
coordinate process
\[
    B_t(\omega):=\omega(t),
    \qquad t\in[0,T],
\]
is an \(n\)-dimensional fractional Brownian motion with Hurst parameter \(H\).
When the process is considered under \(\mathbb P_H\), we write it as
\(B^H=\{B_t^H,t\in[0,T]\}\). More precisely,
\[
    B_t^H
    =
    \bigl(B_t^{H,1},\ldots,B_t^{H,n}\bigr),
\]
where \(B^{H,1},\ldots,B^{H,n}\) are independent real-valued fractional
Brownian motions satisfying
\[
    \mathbb E_{\mathbb P_H}
    \bigl[
        B_t^{H,i}B_s^{H,j}
    \bigr]
    =
    \delta_{ij}R_H(t,s),
    \qquad 1\le i,j\le n,
\]
with
\[
    R_H(t,s)
    =
    \frac12
    \left(
        |t|^{2H}+|s|^{2H}-|t-s|^{2H}
    \right),
    \qquad s,t\in[0,T].
\]
Equivalently, \(\mathbb P_H\) is the product Gaussian measure induced by the
independent components \(B^{H,1},\ldots,B^{H,n}\).

In the special case \(H=\frac12\), the process \(B^{1/2}\) is the standard
\(n\)-dimensional Brownian motion. Accordingly, \(\mathbb P_{1/2}\) is the
standard Wiener measure on \(C_0([0,T];\mathbb R^n)\). For
\(H\neq \frac12\), the measure \(\mathbb P_H\) is the fractional Gaussian
measure associated with the \(n\)-dimensional fractional Brownian motion. In
this case, \(B^H\) is neither a semimartingale nor a Markov process. Moreover,
its sample paths almost surely belong to
\[
    C_0^{H-\varepsilon}([0,T];\mathbb R^n),
    \qquad \varepsilon\in(0,H).
\]
Throughout this paper, we write
\[
    \alpha=\left|H-\frac12\right|.
\]

For the probabilistic arguments below, in particular for the application of
the fractional Girsanov theorem, it is convenient to realize fractional
Brownian motion on a Wiener space. More precisely, we work on a complete
probability space
\[
    (\Omega_W,\mathcal F_W,\mathbb P_{1/2})
\]
carrying an \(n\)-dimensional standard Brownian motion
\(W=\{W_t,t\in[0,T]\}\). Then, for each \(H\in(0,1)\), fractional Brownian
motion can be represented as
\[
    B_t^H
    =
    \int_0^t K_H(t,s)\,dW_s,
    \qquad t\in[0,T],
\]
where \(K_H(t,s)\) is the standard Volterra kernel. In this sense, the law
\(\mathbb P_H\) introduced above is the image measure of
\(\mathbb P_{1/2}\) under the Volterra map \(W\mapsto K_HW\), namely
\[
    \mathbb P_H
    =
    \mathbb P_{1/2}\circ (K_H)^{-1},
\]
where the notation \(K_HW\) stands for the Gaussian process defined by the
above stochastic integral. Thus, although the canonical coordinate process has
law \(\mathbb P_H\) on \(C_0([0,T];\mathbb R^n)\), all calculations involving
stochastic integrals with respect to \(W\) are carried out under the underlying
Wiener measure \(\mathbb P_{1/2}\).

The corresponding Volterra operator \(K_H\) is defined by
\begin{equation}
    (K_H u)(t)
    =
    \int_0^t K_H(t,s)u(s)\,ds,
    \qquad u\in L^2([0,T];\mathbb R^n).
    \label{kh_operator}
\end{equation}
The range
\[
    \mathcal H_H([0,T];\mathbb R^n)
    :=
    K_H\bigl(L^2([0,T];\mathbb R^n)\bigr)
\]
is the Cameron--Martin space associated with the law of \(B^H\). It is
equipped with the inner product
\[
    \langle K_H u,K_H v\rangle_{\mathcal H_H}
    :=
    \langle u,v\rangle_{L^2([0,T];\mathbb R^n)}.
\]
With this inner product, \(\mathcal H_H([0,T];\mathbb R^n)\) is a separable
Hilbert space.

When \(H=\frac12\), one has
\[
    (K_{1/2}u)(t)
    =
    \int_0^t u(s)\,ds.
\]
Therefore the Cameron--Martin space reduces to the classical space
\[
    \mathcal H_{1/2}([0,T];\mathbb R^n)
    =
    \left\{
        h\in AC([0,T];\mathbb R^n):
        h(0)=0, h'\in L^2([0,T];\mathbb R^n)
    \right\},
\]
with norm
\[
    \|h\|_{\mathcal H_{1/2}}
    =
    \|\dot h\|_{L^2([0,T];\mathbb R^n)}.
\]

For general \(H\in(0,1)\), the Cameron--Martin space
\(\mathcal H_H\) admits an explicit characterization through fractional
calculus. In particular, by Nualart~\cite[Lemma 10]{Nualart}, the inverse
operator \(K_H^{-1}\) can be expressed in terms of left-sided
Riemann--Liouville fractional derivatives. More precisely, for
\(h\in\mathcal H_H\), one has
\begin{equation}
    \begin{aligned}
        (K_H)^{-1}h(s)
        &=
        s^{\alpha}
        D_{0^+}^{\alpha}
        \left(
            s^{-\alpha}
            D_{0^+}^{1-2\alpha}h
        \right)(s),
        && H<\frac12,
        \\
        (K_H)^{-1}h(s)
        &=
         h(s)',
        && H=\frac12,
        \\
        (K_H)^{-1}h(s)
        &=
        s^{\alpha}
        D_{0^+}^{\alpha}
        \left(
            s^{-\alpha} h'
        \right)(s),
        && H>\frac12 .
    \end{aligned}
    \label{eq:KH_inverse}
\end{equation}
If \(H<\frac12\) and \(h\) is sufficiently smooth, for instance absolutely
continuous with \(h(0)=0\), then the first formula can be written in the
simpler form
\begin{equation}
    (K_H)^{-1}h(s)
    =
    s^{-\alpha}
    I_{0^+}^{\alpha}
    \left(
        s^{\alpha} h'
    \right)(s),
    \qquad
    \alpha=\frac12-H .
    \label{eq:KH_inverse_H_less_half_smooth}
\end{equation}

Since this paper deals with systems driven by deterministic time-dependent
diffusion coefficients, stochastic integrals with respect to \(B^H\) are
understood in the pathwise sense whenever Young integration is applicable.

\begin{lemma}[Young integration]
\label{young}
Let \(f\in C^\beta([0,T];\mathbb R)\) and
\(g\in C^\gamma([0,T];\mathbb R)\), where
\(\beta,\gamma\in(0,1)\) and \(\beta+\gamma>1\). Then the
Riemann--Stieltjes integral
\[
    \int_0^t f_s\,dg_s
\]
is well defined pathwise for all \(t\in[0,T]\). Moreover, there exists a
constant \(C_{\beta,\gamma,T}>0\) such that, for any
\(0\le s<t\le T\),
\begin{equation}
    \left|
        \int_s^t f_r\,dg_r
        -
        f_s(g_t-g_s)
    \right|
    \le
    C_{\beta,\gamma,T}
    [f]_{\beta}
    [g]_{\gamma}
    |t-s|^{\beta+\gamma}.
    \label{youngint}
\end{equation}
\end{lemma}

The preceding definition extends componentwise to vector-valued and
matrix-valued integrands. In particular, if
\(f\in C^\beta([0,T];\mathbb R^{n\times n})\) and
\(g\in C^\gamma([0,T];\mathbb R^n)\) with \(\beta+\gamma>1\), then
\[
    \int_0^t f_s\,dg_s
\]
is well defined as an \(\mathbb R^n\)-valued Young integral.

To ensure the well-posedness of the degenerate system \eqref{fir} and the applicability of the fractional Girsanov theorem, we impose the following assumptions on the coefficients.
\begin{assumption}[A]\label{ass:A}
We impose the following assumptions on the coefficients of the degenerate system~\eqref{fir}.

\textbf{(1) General conditions.}
\begin{enumerate}
    \item \textbf{Diffusion coefficient $\sigma$:} The matrix $\sigma_t = \text{diag}(\sigma^1_t, \dots, \sigma^n_t)$ is deterministic. Each component satisfies $\sigma^i \in C^1([0,1])$ and there exist constants $m_i, M_i > 0$ such that $0 < m_i \leq \sigma^i_t \leq M_i$ for all $t \in [0,1]$.
    \item \textbf{Stochastic drift $b$:} The function $b: [0,1] \times \mathbb{R}^m \times \mathbb{R}^n \to \mathbb{R}^n$ is continuous in $(t,x,y)$. It is twice continuously differentiable with respect to the spatial variables $(x,y)$, and is globally bounded.
    \item \textbf{Deterministic drift $a$:} The function $a: [0,1] \times \mathbb{R}^m \times \mathbb{R}^n \to \mathbb{R}^m$ is continuously differentiable ($C^1$) with respect to all its variables $(t,x,y)$. Furthermore, it satisfies a global Lipschitz condition with respect to all three variables. That is, there exists a constant $L_a > 0$ such that
    \[
        |a_s(x_1, y_1) - a_t(x_2, y_2)| \leq L_a \big( |s-t| + |x_1-x_2| + |y_1-y_2| \big),
    \]
    for all $s,t \in [0,1]$, $x_1, x_2 \in \mathbb{R}^m$, and $y_1, y_2 \in \mathbb{R}^n$.
\end{enumerate}

\textbf{(2) Additional conditions depending on the Hurst exponent $H \in (1/2, 1)$.}
    In addition to the general conditions, the coefficient $b$ satisfy a global Lipschitz condition with respect to all their variables. Specifically, there exists a constant $L_b > 0$ such that
    \[
       |b_s(x_1, y_1) - b_t(x_2, y_2)| \leq L_b \big( |s-t| + |x_1-x_2| + |y_1-y_2| \big),
    \]
    for all $s,t \in [0,1]$, $x_1,x_2 \in \mathbb{R}^m$, and $y_1,y_2 \in \mathbb{R}^n$.
\end{assumption}

Under Assumption~\ref{ass:A}, we introduce the modified Volterra operator
\(K_H^\sigma\), which incorporates the time-dependent diffusion coefficient
\(\sigma\). For \(u\in L^2([0,T];\mathbb R^n)\), define
\begin{equation}
    (K_H^\sigma u)(t)
    :=
    \int_0^t \sigma_s\,d(K_Hu)(s),
    \qquad t\in[0,T].
    \label{khsigma}
\end{equation}
Since \(\sigma\in C^1([0,T];\mathbb R^{n\times n})\), this integral is well
defined pathwise as a Young integral. Moreover, since \(\sigma_t\) is
uniformly non-degenerate, the operator \(K_H^\sigma\) is injective.

We now define the \(\sigma\)-weighted fractional Cameron--Martin space by
\begin{equation}
    \mathcal H_H^\sigma([0,T];\mathbb R^n)
    :=
    K_H^\sigma
    \bigl(
        L^2([0,T];\mathbb R^n)
    \bigr).
    \label{eq:sigma_weighted_CM_space}
\end{equation}
Equivalently,
\[
    \mathcal H_H^\sigma([0,T];\mathbb R^n)
    =
    \left\{
        h\in C_0([0,T];\mathbb R^n):
        \int_0^\cdot \sigma_s^{-1}\,dh_s
        \in
        \mathcal H_H([0,T];\mathbb R^n)
    \right\}.
\]
The norm on \(\mathcal H_H^\sigma\) is defined by
\begin{equation}
    \|h\|_{\mathcal H_H^\sigma}
    :=
    \|(K_H^\sigma)^{-1}h\|_{L^2([0,T];\mathbb R^n)}.
    \label{eq:sigma_weighted_CM_norm}
\end{equation}
In terms of the standard Cameron--Martin inverse, this can be written as
\begin{equation}
    (K_H^\sigma)^{-1}h
    =
    K_H^{-1}
    \left(
        \int_0^\cdot \sigma_s^{-1}\,dh_s
    \right).
    \label{eq:Ksigma_inverse}
\end{equation}

If \(h\) is sufficiently smooth, then \((K_H^\sigma)^{-1}h\) admits the
following explicit expressions:
\begin{equation}
    (K_H^\sigma)^{-1}h(s)
    =
    \begin{cases}
    \displaystyle
    s^{-\alpha}
    I_{0^+}^{\alpha}
    \left(
        s^{\alpha}\sigma_s^{-1} h(s)'
    \right),
    & H<\frac12,\quad \alpha=\frac12-H,
    \\[1.2em]
    \displaystyle
    \sigma_s^{-1} h(s)',
    & H=\frac12,
    \\[1.2em]
    \displaystyle
    s^{\alpha}
    D_{0^+}^{\alpha}
    \left(
        s^{-\alpha}\sigma_s^{-1} h(s)'
    \right),
    & H>\frac12,\quad \alpha=H-\frac12 .
    \end{cases}
    \label{eq:Ksigma_inverse_explicit}
\end{equation}

The space \(\mathcal H_H^\sigma\) is the natural Cameron--Martin space
associated with the Gaussian process
\[
    \int_0^\cdot \sigma_s\,dB_s^H .
\]
It therefore provides the natural class of admissible skeleton paths for the
system considered in this paper. More precisely, if a target path \(\phi\)
satisfies
\[
    \phi-y_0\in \mathcal H_H^\sigma([0,T];\mathbb R^n),
\]
then there exists a unique control
\[
    \dot{\phi}\in L^2([0,T];\mathbb R^n)
\]
such that
\begin{equation}
    \phi_t-y_0
    =
    (K_H^\sigma  \dot{\phi})(t),
    \qquad t\in[0,T].
    \label{khs}
\end{equation}
Equivalently,
\[
    \dot{\phi}
    =
    (K_H^\sigma)^{-1}(\phi-y_0).
\]
Here \( \dot{\phi}\) should be understood as the Cameron--Martin control associated
with \(\phi\), and not necessarily as the classical time derivative of
\(\phi\) when \(H\neq\frac12\).

Since \(\sigma\) and \(\sigma^{-1}\) are \(C^1\) and bounded, the map
\[
    h
    \mapsto
    \int_0^\cdot \sigma_s\,dh_s
\]
is an isomorphism on the fractional Cameron--Martin space. Consequently,
\(\mathcal H_H^\sigma\) coincides with \(\mathcal H_H\) as a set, but it is
equipped with the equivalent norm induced by \(K_H^\sigma\). In particular,
by the fractional Sobolev embedding, elements of
\(\mathcal H_H^\sigma([0,T];\mathbb R^n)\) belong to
\[
    C_0^\beta([0,T];\mathbb R^n),
    \qquad \beta<H .
\]

\begin{remark}
The boundedness assumptions in Assumption~\ref{ass:A} are mainly imposed to
ensure the validity of the Girsanov transformation and the estimates used
below. We do not aim to pursue optimal regularity assumptions on the
coefficients in this paper.
\end{remark}

\begin{remark}
In the non-degeneracy condition on \(\sigma\), we may assume without loss of
generality that each component \(\sigma^i\) is strictly positive. Indeed, if a
component has the opposite sign, then
\[
    \sigma_t^i\,dB_t^{H,i}
    =
    -\sigma_t^i\,d(-B_t^{H,i}),
\]
and \(-B^{H,i}\) is again a one-dimensional fractional Brownian motion.
\end{remark}

\subsection{Small-ball estimates and the Gaussian correlation inequality}
In this subsection, we present several small-ball estimates that will be used
in the derivation of the OM functional. It is readily verified
that, under the non-degeneracy condition on $\sigma$, the H\"older norm
\[
    \left\|
        \int_0^\cdot \sigma_s\,dB_s^H
    \right\|_\beta
\]
is a measurable norm on the Cameron--Martin space in the sense of
\cite{Nualart}. Consequently, the corresponding estimates in \cite{Nualart}
can be extended to the present time-dependent setting.

For $\varepsilon>0$, we denote the H\"older small-ball event by
\[
    A_\varepsilon
    :=
    \left\{
    \left\|
        \int_0^\cdot \sigma_s\,dB_s^H
    \right\|_\beta
    <
    \varepsilon
    \right\}.
\]

\begin{lemma}\label{estimate1}
Let $0<\beta<H$. Then, for any $h\in L^2([0,1])$, one has
\begin{equation}
    \lim_{\varepsilon\to0}
    \mathbb{E}
    \left[
        \exp\left(
            \int_0^1 h(s)\,dW_s
        \right)
        \,\middle|\,
        A_\varepsilon
    \right]
    =
    1 .
\end{equation}
\end{lemma}

We also need a second-order version of the preceding estimate. Let
$f\in L^2([0,1]^2)$ be symmetric. The Hilbert--Schmidt operator associated
with $f$ is defined by
\begin{equation}\label{eq:Kf}
    (K(f)h)(t)
    :=
    \int_0^1 f(t,u)h(u)\,du,
    \qquad h\in L^2([0,1]).
\end{equation}
When $K(f)$ is trace class and $f$ is continuous, its trace is given by
\[
    \operatorname{Tr}K(f)
    =
    \int_0^1 f(s,s)\,ds .
\]
For notational simplicity, we write
\[
    \operatorname{Tr}f
    :=
    \operatorname{Tr}K(f).
\]

\begin{lemma}\label{estimate2}
Let $f\in L^2([0,1]^2)$ be symmetric. Assume that $K(f)$ is trace class and
that $0<\beta<H$. Then
\[
    \lim_{\varepsilon\to0}
    \mathbb{E}
    \left[
        \exp\left(
            \int_0^1\int_0^1 f(s,t)\,dW_s\,dW_t
        \right)
        \,\middle|\,
        A_\varepsilon
    \right]
    =
    \exp\bigl(-\operatorname{Tr}f\bigr).
\]
\end{lemma}

In addition, we shall use the following exponential integrability estimate.

\begin{lemma}\label{estimate1'}
Let $0<\beta<H$. Then, for any $h\in L^2([0,1])$, one has
\[
    \lim_{\varepsilon\to0}
    \mathbb{E}
    \left[
        \exp\left(
            \left|
                \int_0^1 h(s)\,dW_s
            \right|
        \right)
        \,\middle|\,
        A_\varepsilon
    \right]
    =
    1 .
\]
\end{lemma}

\begin{proof}
For simplicity, write
\[
    \mathbb{E}^\varepsilon[\cdot]
    :=
    \mathbb{E}[\cdot\,|\,A_\varepsilon],
    \qquad
    U
    :=
    \int_0^1 h(s)\,dW_s .
\]
By Lemma~\ref{estimate1},
\[
    \lim_{\varepsilon\to0}
    \mathbb{E}^\varepsilon\bigl(e^U\bigr)
    =
    1 .
\]
Since the event $A_\varepsilon$ is symmetric and $U$ is an odd linear functional
of the underlying Wiener path, the conditional distribution of $U$ under
$A_\varepsilon$ is symmetric. Hence
\[
    \mathbb{E}^\varepsilon\bigl(e^U\bigr)
    =
    \mathbb{E}^\varepsilon\bigl(e^{-U}\bigr),
\]
and therefore
\[
    \lim_{\varepsilon\to0}
    \mathbb{E}^\varepsilon\bigl(e^{-U}\bigr)
    =
    1 .
\]
It follows that
\[
\begin{aligned}
    2
    &=
    \lim_{\varepsilon\to0}
    \mathbb{E}^\varepsilon
    \bigl(e^U+e^{-U}\bigr)  \\
    &=
    \lim_{\varepsilon\to0}
    \mathbb{E}^\varepsilon
    \bigl(e^{|U|}+e^{-|U|}\bigr).
\end{aligned}
\]
Set
\[
    A
    :=
    \lim_{\varepsilon\to0}
    \mathbb{E}^\varepsilon\bigl(e^{|U|}\bigr).
\]
By Cauchy's inequality,
\[
    \mathbb{E}^\varepsilon\bigl(e^{|U|}\bigr)
    \mathbb{E}^\varepsilon\bigl(e^{-|U|}\bigr)
    \geq
    \left(
        \mathbb{E}^\varepsilon[1]
    \right)^2
    =
    1 .
\]
Passing to the limit gives
\[
    A(2-A)\geq 1.
\]
Equivalently,
\[
    (A-1)^2\leq 0.
\]
Thus $A=1$, which proves the claim.
\end{proof}

Analogous to the small-ball probability estimates of the H\"{o}lder norm for the standard fractional Brownian motion, a similar estimate holds for the multidimensional time-varying setting; the reader is referred to \cite{KuelbsLiShao1995} for a detailed proof.

\begin{theorem}
Let $\sigma$ satisfy Assumption~(A). Then there exists a positive constant $M_{\beta,H}$ such that for any $\varepsilon > 0$,
\begin{equation}
\mathbb{P}\left( A_\varepsilon \right) \geq \exp\left( -M_{\beta,H} \, n \, \varepsilon^{-\frac{1}{H-\beta}} \right).
\end{equation}
\end{theorem}
The proof of this theorem relies primarily on estimating the variance of $\int_0^t \sigma_s \, dB^H_s$. Since $\sigma_s$ is deterministic, the upper and lower bounds of this estimate depend solely on the time interval and the regularity of $\sigma_s$. For more general stochastic processes, however, additional terms arise due to randomness (see \cite{MaayanMayerWolf2018}).

The technical cornerstone for proving the above estimates is the GCI. Long a celebrated conjecture in Gaussian geometry, it was finally settled by Royen \cite{royen2014convex} using a remarkably elegant approach. Its validity in general Gaussian spaces has been further elucidated by Lata{\l}a and Matlak \cite{latala2017royen}.

\begin{theorem}[Gaussian Correlation Inequality]
Let $\mu$ be a centered Radon Gaussian measure on a separable Banach space $E$. For any two $\mu$-measurable, symmetric, and convex sets $C_1, C_2 \subseteq E$, the following inequality holds:
\begin{equation} \label{GCI}
\mu(C_1 \cap C_2) \geq \mu(C_1) \mu(C_2).
\end{equation}
\end{theorem}

\begin{remark}
While the Gaussian Correlation Inequality \eqref{GCI} was originally proved by Royen \cite{royen2014convex} (and further clarified by Lata{\l}a and Matlak \cite{latala2017royen}) for standard Gaussian measures on finite-dimensional Euclidean spaces $\mathbb{R}^d$, its extension to any separable Banach space $E$ equipped with a centered Radon Gaussian measure $\mu$ is a standard consequence of measure-theoretic approximation.

Specifically, by the  Hahn-Banach theorem, any closed, symmetric, and convex set $C \subseteq E$ is weakly closed and can be represented as the countable intersection of symmetric strips:
\[
    C = \bigcap_{i=1}^\infty \left\{ \omega \in E : |l_i(\omega)| \leq c_i \right\},
\]
where $\{l_i\}_{i=1}^\infty \subset E^*$ is a countable separating family of continuous linear functionals and $c_i > 0$. For any integer $n \geq 1$, define the finite-dimensional cylindrical sets
\[
    C^{(n)}_1 = \bigcap_{i=1}^n \left\{ \omega \in E : |l^{(1)}_i(\omega)| \leq c^{(1)}_i \right\} \quad \text{and} \quad C^{(n)}_2 = \bigcap_{i=1}^n \left\{ \omega \in E : |l^{(2)}_i(\omega)| \leq c^{(2)}_i \right\}.
\]
Since $C^{(n)}_1$ and $C^{(n)}_2$ depend only on the finite-dimensional projections of the Gaussian measure $\mu$, the finite-dimensional GCI is directly applicable, yielding $\mu\big(C^{(n)}_1 \cap C^{(n)}_2\big) \geq \mu\big(C^{(n)}_1\big) \mu\big(C^{(n)}_2\big)$.
Because the sequences of sets $\big\{C^{(n)}_1\big\}_{n=1}^\infty$ and $\big\{C^{(n)}_2\big\}_{n=1}^\infty$ are monotonically decreasing with $C_k = \bigcap_{n=1}^\infty C^{(n)}_k$ for $k=1,2$, the continuity of the Radon measure $\mu$ from above guarantees that the inequality passes to the limit as $n \to \infty$. The result for arbitrary $\mu$-measurable symmetric convex sets then follows by inner regularity.
\end{remark}

\begin{remark}
In our context, the $\varepsilon$-tube in the H\"{o}lder space, defined by $A_\varepsilon = \{ \omega : \|\int_0^\cdot \sigma_u dB^H_u(\omega)\|_\beta < \varepsilon \}$, constitutes a symmetric convex set. 
\end{remark}

\begin{corollary} \label{gauss set}
Let $(\Omega, \mathcal{F}, \mathbb{P})$ be an abstract Wiener space (where $\Omega$ is a separable Banach space and $\mathbb{P}$ is the centered standard Gaussian measure on $\Omega$). 
Let $A \subseteq \Omega$ be a $\mu$-measurable, symmetric, and convex set, and let $f: \Omega \to [0, \infty)$ be a $\mathbb{P}$-measurable, symmetric, and convex function. Then for every $t > 0$, the following estimate holds:
\begin{equation*}
    \mathbb{P}\bigl(\{x \in \Omega : f(x) \le t\} \cap A\bigr)
    \ge
    \mathbb{P}\bigl(\{x \in \Omega: f(x) \le t\}\bigr) \mathbb{P}(A).
\end{equation*}
\end{corollary}

\begin{proof}
By the assumptions on $f$, the lower level set $C_t := \{x \in \Omega : f(x) \le t\}$ is $\mu$-measurable, symmetric, and convex in $\Omega$ for any given $t > 0$. Since $\Omega$ is a separable Banach space and $\mathbb{P}$ is a centered Gaussian measure on $\Omega$, we can directly apply the Gaussian Correlation Inequality \eqref{GCI} to the sets $C_t$ and $A$. This immediately yields
\begin{equation*}
    \mathbb{P}(C_t \cap A) \ge \mathbb{P}(C_t) \mathbb{P}(A),
    \end{equation*}
which completes the proof.
\end{proof}

\begin{remark}
It is worth noting that in the absence of a general proof for \eqref{GCI} at the time, \cite[Corollary 4.6.3]{Bogachev1998} established the above inequality only for the specific case where the level sets are defined by linear functionals. The breakthrough by Royen \cite{royen2014convex} allows us to extend this result to the general convex setting presented here.
\end{remark}

\begin{lemma}[Layer Cake Representation]\label{tail}
Let $Y$ be a non-negative random variable. Then
\begin{equation*}
    \mathbb{E}[Y] = \int_0^\infty \mathbb{P}(Y > s) \, ds.
\end{equation*}
\end{lemma}

The following theorem demonstrates that under the symmetry and convexity assumptions, the conditional expectation over a symmetric convex set is dominated by the unconditional expectation.

\begin{theorem}\label{conditional-expectation}
Let $f$ be a non-negative, convex, symmetric, and measurable function on the Wiener space $(\Omega, \mathcal{F}, \mathbb{P})$. Let $A \in \mathcal{F}$ be a symmetric convex set with $\mathbb{P}(A) > 0$. Then the following inequality holds:
\begin{equation*}
    \mathbb{E}\bigl[f \mid A\bigr] \le \mathbb{E}\bigl[f\bigr].
\end{equation*}
\end{theorem}

\begin{proof}
For each $t \ge 0$, we consider the level set $S_t := \{\omega \in \Omega : f(\omega) \le t\}$. Since $f$ is assumed to be symmetric and convex, $S_t$ is a symmetric convex measurable set in $\Omega$. Applying Corollary \ref{gauss set} (with $\mu = \mathbb{P}$ and $X = \Omega$), we have
\begin{equation*}
    \mathbb{P}(S_t \cap A) \ge \mathbb{P}(S_t) \mathbb{P}(A).
\end{equation*}
By the properties of the probability measure restricted to the set $A$, we observe that
\begin{align*}
    \mathbb{P}(S_t^c \cap A) &= \mathbb{P}(A) - \mathbb{P}(S_t \cap A) \\
    &\le \mathbb{P}(A) - \mathbb{P}(S_t) \mathbb{P}(A) \\
    &= (1 - \mathbb{P}(S_t)) \mathbb{P}(A) = \mathbb{P}(S_t^c) \mathbb{P}(A).
\end{align*}
Dividing both sides by $\mathbb{P}(A)$ yields the comparison for the conditional tail probabilities:
\begin{equation*}
    \mathbb{P}(f > t \mid A) = \frac{\mathbb{P}(S_t^c \cap A)}{\mathbb{P}(A)} \le \mathbb{P}(S_t^c) = \mathbb{P}(f > t).
\end{equation*}
Integrating the above tail probability with respect to $t$ and invoking Lemma \ref{tail}, we conclude that
\begin{align*}
    \mathbb{E}\bigl[f \mid A \bigr] &= \int_0^\infty \mathbb{P}(f > t \mid A) \, dt \\
    &\le \int_0^\infty \mathbb{P}(f > t) \, dt \\
    &= \mathbb{E}\bigl[f\bigr].
\end{align*}
This completes the proof.
\end{proof}

  \section{ Main Results}
  In this section, we present the  main results. 
  We begin with several preliminary lemmas involving the Girsanov transformation, H\"{o}lder estimates for the drift components, and the regularity of fractional operators.
  
 \subsection{Simplification via Girsanov Transformation}
First, the OM functional in \eqref{omf} can be simplified by applying the Girsanov transformation. Here, we define the underlying product probability space as $\Omega = \Omega_1 \times \dots \times \Omega_n$, where each $\Omega_i$ denotes the Wiener space corresponding to the standard Brownian motion that underlies the $i$-th fractional Brownian motion. The overall filtration $\mathcal{F}_t$ and the $\sigma$-algebra $\mathcal{F}$ are generated by the component filtrations $\mathcal{F}^i_t$ and $\mathcal{F}^i$, respectively. The relevant transformations for the fractional and degenerate cases are discussed in detail in \cite{Nualart} and \cite{liu2024onsager}, respectively. 

For any reference path $\phi$ satisfying $\phi - y_0 \in \mathcal{H}_H$, let us define the processes $\tilde{Y}$ and $\tilde{X}$ by
\begin{equation} \label{trans}
\tilde{Y}_t = \phi_t + \int_0^t \sigma_s \, dB^H_s, \quad \tilde{X}_t = x_0 + \int_0^t a_s(\tilde{X}_s, \tilde{Y}_s) \, ds.
\end{equation}
We introduce the drift process $u_s$ defined as
\begin{equation}
u_s = \dot{\phi}_s - \left(K_H^\sigma\right)^{-1} \left( \int_0^{\cdot} b_u(\tilde{X}_u, \tilde{Y}_u) \, du \right)(s),
\end{equation}
and consider the shifted standard Brownian motion
\begin{equation}
\tilde{W}_t = W_t + \int_0^t u_s \, ds.
\end{equation}
Under the assumptions stipulated in this paper, the Novikov condition is satisfied (for a detailed proof, the reader is referred to \cite{NUALART2002103}). Consequently, by Girsanov's theorem, there exists a probability measure $\tilde{\mathbb{P}}$ that is absolutely continuous with respect to $\mathbb{P}$, under which $\tilde{W}_t$ is a standard Brownian motion. The corresponding fractional Brownian motion associated with $\tilde{W}_t$ is given by
\begin{equation*}
\tilde{B}^H_t = B^H_t + \int_0^t \sigma_s^{-1} \, d\phi_s - \int_0^t \sigma_s^{-1} b_s(\tilde{X}_s, \tilde{Y}_s) \, ds.
\end{equation*}
Transforming to the differential form under the new measure $\tilde{\mathbb{P}}$, the coupled dynamics of $(\tilde{X}, \tilde{Y})$ become
\begin{equation} \label{firt}
\begin{cases}
d\tilde{X}_t = a_t(\tilde{X}_t, \tilde{Y}_t)\, dt, \\
d\tilde{Y}_t = b_t(\tilde{X}_t, \tilde{Y}_t) \, dt + \sigma_t \, d\tilde{B}^H_t.
\end{cases}
\end{equation}
This implies that $(\tilde{X}, \tilde{Y}, \tilde{B}^H)$ is the unique strong solution to \eqref{fir} under $\tilde{\mathbb{P}}$.

Consequently, the small-ball probability ratio appearing in the OM functional can be evaluated by leveraging the Radon--Nikodym derivative $\frac{d\tilde{\mathbb{P}}}{d\mathbb{P}}$:
\begin{equation}\label{simplification}
\begin{aligned}
&\frac{\mathbb{P}\bigl(\|Y-\phi\|_{\beta}\le \varepsilon\bigr)}
{\mathbb{P}\!\left( \left\|\int_0^\cdot \sigma_u\, dB_u^H\right\|_{\beta}\le \varepsilon\right)} \\
&\frac{\tilde{\mathbb{P}}\bigl(\|\tilde{Y}-\phi\|_{\beta}\le \varepsilon\bigr)}
{\mathbb{P}\!\left( \left\|\int_0^\cdot \sigma_u\, dB_u^H\right\|_{\beta}\le \varepsilon\right)} \\
&= \frac{\tilde{\mathbb{P}}\!\left( \left\|\int_0^\cdot \sigma_u\, dB_u^H\right\|_{\beta}\le \varepsilon \right)}
{\mathbb{P}\!\left(\left\|\int_0^\cdot \sigma_u\, dB_u^H\right\|_{\beta}\le \varepsilon\right)} \\
&= \frac{\mathbb{E}\!\left[ \exp\!\left( -\int_0^1 u_s \, dW_s - \frac12\int_0^1 |u_s|^2\, ds \right) \mathbf{1}_{\left\{\left\|\int_0^\cdot \sigma_u\, dB_u^H\right\|_{\beta}\le \varepsilon\right\}} \right]}
{\mathbb{P}\!\left(\left\|\int_0^\cdot \sigma_u\, dB_u^H\right\|_{\beta}\le \varepsilon\right)} \\
&= \mathbb{E}\!\left[
\exp\!\left(
-\int_0^1 u_s \, dW_s
-\frac12\int_0^1 |u_s|^2\, ds
\right)
\,\middle|\,
\left\|\int_0^\cdot \sigma_u\, dB_u^H\right\|_{\beta}\le \varepsilon
\right].
\end{aligned}
\end{equation}

\subsection{Gronwall Estimates and Fractional Regularity}

We next employ Gronwall's inequality to show that, under the small-ball constraint, the deviation $\|\tilde{X}-\psi\|_{1+\beta}$ is of order $\varepsilon$.

    \begin{lemma}\label{xholder}
Let $\phi-y_0\in \mathcal{H}_H^\sigma([0,1];\mathbb{R}^n)$, and let
$(\psi_t,\phi_t)_{t\in[0,1]}$ be the $\mathbb{R}^m\times\mathbb{R}^n$-valued
reference paths satisfying \eqref{sec}. Let $(\tilde X_t,\tilde Y_t)_{t\in[0,1]}$
be the strong solution to \eqref{trans}. Assume that
\[
    \left\|\int_0^\cdot \sigma_u\,dB_u^H\right\|_\beta \le \varepsilon .
\]
Assume further that $a$ is globally Lipschitz in $(x,y)$ and that its first
spatial derivatives are uniformly $\beta$-H\"older continuous in $(t,x,y)$ on
bounded sets. Then there exists a constant $C>0$, independent of $\varepsilon$,
such that
\begin{equation}\label{grownwall_highdim}
    \|\tilde X-\psi\|_{1+\beta}\le C\varepsilon .
\end{equation}
\end{lemma}

\begin{proof}
Set
\[
    Z_t:=\tilde X_t-\psi_t,\qquad
    \mathcal B_t:=\int_0^t \sigma_u\,dB_u^H .
\]
Then $\tilde Y_t-\phi_t=\mathcal B_t$ and $Z_0=0$. By the norm convention on
$C^{1+\beta}_0([0,1];\mathbb R^m)$, it suffices to prove
\[
    \sup_{t\in[0,1]} |Z_t|\le C\varepsilon
\]
and
\[
    [Z']_\beta\le C\varepsilon .
\]

We first prove the uniform estimate. From the equations for $\tilde X$ and
$\psi$, we have
\[
    Z_t'
    =
    a_t(\tilde X_t,\tilde Y_t)-a_t(\psi_t,\phi_t)
    =
    a_t(\psi_t+Z_t,\phi_t+\mathcal B_t)-a_t(\psi_t,\phi_t).
\]
By the global Lipschitz continuity of $a$ in $(x,y)$,
\[
    |Z_t'|
    \le C\bigl(|Z_t|+|\mathcal B_t|\bigr).
\]
Since the path $\mathcal B$ starts from zero and
$\|\mathcal B\|_\beta\le\varepsilon$, we have
\[
    \sup_{t\in[0,1]}|\mathcal B_t|\le C\varepsilon .
\]
Hence
\[
    |Z_t|
    \le C\int_0^t |Z_s|\,ds + C\varepsilon .
\]
Gronwall's inequality gives
\[
    \sup_{t\in[0,1]} |Z_t|\le C\varepsilon .
\]
Substituting this estimate into the bound for $Z_t'$ yields
\[
    \sup_{t\in[0,1]} |Z_t'|\le C\varepsilon .
\]

It remains to estimate the $\beta$-H\"older seminorm of $Z'$. For
$0\le s<t\le 1$, by the integral form of the mean value theorem in
$\mathbb R^m\times\mathbb R^n$, we have
\[
\begin{aligned}
Z_t'
&=
\int_0^1
\Big[
    \nabla_x a_t(\psi_t+\lambda Z_t,\phi_t+\lambda\mathcal B_t)Z_t
    +
    \nabla_y a_t(\psi_t+\lambda Z_t,\phi_t+\lambda\mathcal B_t)\mathcal B_t
\Big]\,d\lambda .
\end{aligned}
\]
Consequently,
\[
\begin{aligned}
Z_t'-Z_s'
&=
\int_0^1
\Big[
    A_x(t,\lambda)Z_t-A_x(s,\lambda)Z_s
\Big]\,d\lambda
\\
&\quad+
\int_0^1
\Big[
    A_y(t,\lambda)\mathcal B_t-A_y(s,\lambda)\mathcal B_s
\Big]\,d\lambda ,
\end{aligned}
\]
where
\[
    A_x(r,\lambda)
    :=
    \nabla_x a_r(\psi_r+\lambda Z_r,\phi_r+\lambda\mathcal B_r),
\]
and
\[
    A_y(r,\lambda)
    :=
    \nabla_y a_r(\psi_r+\lambda Z_r,\phi_r+\lambda\mathcal B_r).
\]

We estimate the $x$-part first. Write
\[
\begin{aligned}
A_x(t,\lambda)Z_t-A_x(s,\lambda)Z_s
&=
A_x(t,\lambda)(Z_t-Z_s)
+
\bigl(A_x(t,\lambda)-A_x(s,\lambda)\bigr)Z_s .
\end{aligned}
\]
Since $\sup_{r\in[0,1]}|Z_r'|\le C\varepsilon$, we have
\[
    |Z_t-Z_s|\le C\varepsilon |t-s|\le C\varepsilon |t-s|^\beta .
\]
Moreover, by the assumed $\beta$-H\"older regularity of $\nabla a$ and by the
regularity of $\psi$, $\phi$, $Z$, and $\mathcal B$, we obtain
\[
\begin{aligned}
\|A_x(t,\lambda)-A_x(s,\lambda)\|
&\le C\Big(
    |t-s|^\beta
    +|\psi_t-\psi_s|^\beta
    +|\phi_t-\phi_s|^\beta
    +|Z_t-Z_s|^\beta
    +|\mathcal B_t-\mathcal B_s|^\beta
\Big)
\\
&\le C|t-s|^\beta .
\end{aligned}
\]
Together with $\sup_{r\in[0,1]}|Z_r|\le C\varepsilon$, this gives
\[
    |A_x(t,\lambda)Z_t-A_x(s,\lambda)Z_s|
    \le C\varepsilon |t-s|^\beta .
\]

The $y$-part is treated in the same way:
\[
\begin{aligned}
A_y(t,\lambda)\mathcal B_t-A_y(s,\lambda)\mathcal B_s
&=
A_y(t,\lambda)(\mathcal B_t-\mathcal B_s)
+
\bigl(A_y(t,\lambda)-A_y(s,\lambda)\bigr)\mathcal B_s .
\end{aligned}
\]
Since $\|\mathcal B\|_\beta\le \varepsilon$, we have
\[
    |\mathcal B_t-\mathcal B_s|
    \le \varepsilon |t-s|^\beta,
    \qquad
    \sup_{r\in[0,1]}|\mathcal B_r|\le C\varepsilon .
\]
Using again the $\beta$-H\"older regularity of $\nabla a$, we obtain
\[
    |A_y(t,\lambda)\mathcal B_t-A_y(s,\lambda)\mathcal B_s|
    \le C\varepsilon |t-s|^\beta .
\]
Combining the preceding estimates and integrating over $\lambda\in[0,1]$, we
conclude that
\[
    |Z_t'-Z_s'|
    \le C\varepsilon |t-s|^\beta .
\]
Therefore
\[
    [Z']_\beta\le C\varepsilon .
\]
Together with the uniform bound on $Z'$, this proves
\[
    \|\tilde X-\psi\|_{1+\beta}
    =
    \|Z\|_{1+\beta}
    \le C\varepsilon .
\]
The proof is complete.
\end{proof}

    \begin{lemma}\label{regular-c1}
Let $H>\frac12$ and set $\alpha=H-\frac12\in(0,\frac12)$. Let
$f\in C^{1+\beta}_0([0,1];\mathbb R^d)$ with
\[
    \alpha<\beta<\alpha+\frac14 .
\]
Then the mapping
\[
    s\mapsto s^{2\alpha}D_{0^+}^{\alpha}(s^{-\alpha}f_s)
\]
belongs to $C^1([0,1];\mathbb R^d)$. Moreover, there exists a constant $C>0$
such that
\begin{equation}\label{eq:reg_c1_bound_vector}
    \left\|
        s^{2\alpha}D_{0^+}^{\alpha}(s^{-\alpha}f_s)
    \right\|_{C^1([0,1];\mathbb R^d)}
    \le C\|f\|_{1+\beta}.
\end{equation}
\end{lemma}

\begin{proof}
The fractional derivative is understood componentwise. For each component
$f^i$, $1\le i\le d$, the scalar argument gives
\[
    s\mapsto s^{2\alpha}D_{0^+}^{\alpha}(s^{-\alpha}f_s^i)
    \in C^1([0,1]),
\]
and
\[
    \left\|
        s^{2\alpha}D_{0^+}^{\alpha}(s^{-\alpha}f_s^i)
    \right\|_{C^1}
    \le C\|f^i\|_{1+\beta}.
\]
Taking the maximum over $1\le i\le d$ yields
\[
    \left\|
        s^{2\alpha}D_{0^+}^{\alpha}(s^{-\alpha}f_s)
    \right\|_{C^1([0,1];\mathbb R^d)}
    \le C\|f\|_{1+\beta}.
\]
The proof is complete.
\end{proof}

\begin{lemma}\label{singular-c1}
Let $\frac14<H<\frac12$ and set $\alpha=\frac12-H\in(0,\frac14)$. Let
$f\in C^{1+\beta}_0([0,1];\mathbb R^d)$ with
\[
    0<\beta<H-\frac14 .
\]
Then the mapping
\[
    s\mapsto s^{-2\alpha}I_{0^+}^{\alpha}(s^\alpha f_s)
\]
belongs to $C^1([0,1];\mathbb R^d)$. Moreover, there exists a constant $C>0$
such that
\begin{equation}\label{eq:singular_c1_bound_vector}
    \left\|
        s^{-2\alpha}I_{0^+}^{\alpha}(s^\alpha f_s)
    \right\|_{C^1([0,1];\mathbb R^d)}
    \le C\|f\|_{1+\beta}.
\end{equation}
\end{lemma}

\begin{proof}
Again, the fractional integral is understood componentwise. For each component
$f^i$, using the definition of the Riemann--Liouville fractional integral and
the change of variables $r=sv$, we have
\[
\begin{aligned}
s^{-2\alpha}I_{0^+}^{\alpha}(s^\alpha f_s^i)
&=
\frac{s^{-2\alpha}}{\Gamma(\alpha)}
\int_0^s (s-r)^{\alpha-1}r^\alpha f_r^i\,dr
\\
&=
\frac{1}{\Gamma(\alpha)}
\int_0^1 (1-v)^{\alpha-1}v^\alpha f_{sv}^i\,dv .
\end{aligned}
\]
Since $f^i\in C^{1+\beta}_0([0,1])$, differentiation under the integral sign is
justified by the dominated convergence theorem, and
\[
\frac{d}{ds}
\left(
    s^{-2\alpha}I_{0^+}^{\alpha}(s^\alpha f_s^i)
\right)
=
\frac{1}{\Gamma(\alpha)}
\int_0^1 (1-v)^{\alpha-1}v^{\alpha+1}(f^i)'_{sv}\,dv .
\]
The kernel $(1-v)^{\alpha-1}v^{\alpha+1}$ is integrable on $[0,1]$, and hence
\[
    \left\|
        s^{-2\alpha}I_{0^+}^{\alpha}(s^\alpha f_s^i)
    \right\|_{C^1}
    \le C\|f^i\|_{1+\beta}.
\]
Taking the maximum over all components gives
\[
    \left\|
        s^{-2\alpha}I_{0^+}^{\alpha}(s^\alpha f_s)
    \right\|_{C^1([0,1];\mathbb R^d)}
    \le C\|f\|_{1+\beta}.
\]
The proof is complete.
\end{proof}

    \subsection{The Onsager-Machlup Functional}
   In the following, we present the precise statement and proof of the theorem corresponding to \eqref{mainresultfunctional}. To unify the presentation across different Hurst regimes, the Onsager--Machlup functional will be expressed in terms of the operator $(K_H^\sigma)^{-1}$.
\begin{theorem}\label{result}
Let $H \in (1/4, 1)$, and let $(X,Y)$ be the solution to the stochastic system \eqref{fir} with coefficients satisfying Assumption \((A)\). For any reference path $\phi-y_0 \in \mathcal{H}^\sigma_H([0,1];\mathbb{R}^n)$ and its associated state $\psi$ satisfying the coupling \eqref{sec} with $\psi_0 = x_0$, the OM functional of $Y$ with respect to the H\"{o}lder norm $\Vert \cdot \Vert_{\beta}$ (with $\max\{H-1/2,0\} < \beta < H-1/4$) is given by
\begin{equation}
    J(\phi) = \frac{1}{2}\int_0^1 \left| (K_H^\sigma)^{-1} \left( \phi_\cdot - y_0 - \int_0^\cdot b_v(\psi_v, \phi_v) \, dv \right)(s) \right|^2 ds + \frac{d_H}{2} \int_0^1 \, \nabla_y \cdot b_s(\psi_s, \phi_s) \, ds, \label{omexpress}
\end{equation}
where the constant $d_H= \sqrt{\frac{2H\Gamma(H+1/2)\Gamma(3/2-H)}{\Gamma(2-2H)}}$.
\end{theorem}

\begin{remark}
Note that although the OM functional in \eqref{simplification} is formulated in terms of the process $Y$, it effectively characterizes the most probable transition path for the entire coupled system $(X, Y)$. This is due to the fact that $X$ is a functional of $Y$ via the deterministic relation \eqref{firt}. As shown in Lemma \ref{xholder}, the small-ball constraint on $Y$ inherently controls the trajectory of $X$, ensuring that the derived $J(\phi)$ is the appropriate Lagrangian for the joint dynamics.
\end{remark}

\begin{proof}
Since the inverse operator $(K_H^\sigma)^{-1}$ admits three distinct representations according to the value of $H$, namely as a fractional derivative, an ordinary derivative, or a fractional integral, the proof is naturally divided into three cases.

\medskip

\noindent\textbf{Case I: $1/4<H<1/2$.}

According to the explicit representation of the inverse operator $(K_H^\sigma)^{-1}$, it suffices to verify that the OM functional is given by
\begin{equation*}
J(\phi) = \frac12\int_0^1 \left| \dot{\phi}_s - s^{-\alpha}I_{0^+}^{\alpha}\!\left(s^\alpha \sigma_s^{-1}b_s(\psi_s,\phi_s)\right) \right|^2 ds + \frac{d_H}{2} \int_0^1 \nabla_y \cdot b_s(\psi_s,\phi_s) \, ds.
\end{equation*}
For brevity, let $A^\varepsilon_k := \left\{ \left\| \int_0^\cdot \sigma^k_u \, dB^{H,k}_u \right\|_\beta \le \varepsilon \right\}$ denote the small-ball event for the $k$-th coordinate component, and define the joint small-ball event as $A^\varepsilon := \bigcap_{k=1}^d A^\varepsilon_k = \left\{\left\|\int_0^\cdot \sigma_s\,dB_s^H\right\|_\beta\le\varepsilon\right\}$. 
Arguments identical to those in the classical Brownian setting are omitted here; we refer the reader to \cite{Liu2026,Nualart} for extensive details. After invoking Girsanov's theorem, the remaining verification follows a standard scheme.

By Girsanov's theorem, the small-ball probability under the transformed measure satisfies
\begin{align*}
\mathbb{P}\bigl(\|\tilde{Y}-\phi\|_{\beta}\le\varepsilon\bigr)
&= \mathbb{E}\Bigg[ \exp\Bigg( -\int_0^1 u_s \cdot dW_s -\frac12\int_0^1 |u_s|^2\,ds \Bigg) \mathbf{1}_{A^\varepsilon} \Bigg] \\
&= \mathbb{E}\Big( \exp(\mathcal{I}_1+\mathcal{I}_2)\, \mathbf{1}_{A^\varepsilon} \Big),
\end{align*}
where the exponents are partitioned as
\begin{align*}
\mathcal{I}_1 &= -\int_0^1 \left( \dot{\phi}_s - s^{-\alpha}I_{0^+}^{\alpha}\!\left(s^\alpha\sigma_s^{-1}b_s(\tilde X_s,\tilde Y_s)\right) \right) \cdot dW_s,\\
\mathcal{I}_2 &= -\frac12\int_0^1 \left| \dot{\phi}_s - s^{-\alpha}I_{0^+}^{\alpha}\!\left(s^\alpha\sigma_s^{-1}b_s(\tilde X_s,\tilde Y_s)\right) \right|^2 ds.
\end{align*}
By virtue of Lemma~\ref{xholder}, it follows immediately that as the noise path vanishes under the H\"{o}lder topology,
\begin{equation*}
I_{0^+}^{\alpha}\!\left(s^\alpha\sigma_s^{-1}b_s(\tilde X_s,\tilde Y_s)\right) \to I_{0^+}^{\alpha}\!\left(s^\alpha\sigma_s^{-1}b_s(\psi_s,\phi_s)\right), \qquad \text{as } \varepsilon \to 0.
\end{equation*}
Consequently, applying the Lebesgue dominated convergence theorem yields
\begin{align*}
\lim_{\varepsilon\to0} \mathbb{E}\Bigg( \exp(\mathcal{I}_2)\,\Bigg|\, A^\varepsilon \Bigg) = \exp\Bigg( -\frac12\int_0^1 \left| \dot{\phi}_s- s^{-\alpha}I_{0^+}^{\alpha}\!\left(s^\alpha\sigma_s^{-1}b_s(\psi_s,\phi_s)\right) \right|^2 ds \Bigg).
\end{align*}
Moreover, by Lemma~\ref{estimate1} and the coordinate-wise independence of the underlying Brownian motions, the component-wise decoupling holds:
\begin{equation*}
\begin{aligned}
\lim_{\varepsilon\to0} \mathbb{E}\Bigg( \exp\Bigl(-\int_0^1 \dot{\phi}_s \cdot dW_s\Bigr) \,\Bigg|\, A^\varepsilon \Bigg) 
&= \lim_{\varepsilon\to0} \mathbb{E}\Bigg( \prod_{i=1}^d \exp\Bigl(-\int_0^1 \dot{\phi}_s^i\, dW^i_s\Bigr) \,\Bigg|\, \max_{1 \le i \le d} \left\| \int_0^\cdot \sigma^i_u\,dB_u^{H,i}\right\|_\beta\le\varepsilon \Bigg) \\
&= \lim_{\varepsilon\to0} \mathbb{E}\Bigg( \prod_{i=1}^d \exp\Bigl(-\int_0^1 \dot{\phi}_s^i\, dW^i_s\Bigr) \,\Bigg|\, \bigcap_{i=1}^d A^\varepsilon_i \Bigg) \\
&= \lim_{\varepsilon\to0} \prod_{i=1}^d \mathbb{E}\Bigg( \exp\Bigl(-\int_0^1 \dot{\phi}_s^i\, dW^i_s\Bigr) \,\Bigg|\, A^\varepsilon_i \Bigg) \\
&= \prod_{i=1}^d \lim_{\varepsilon\to0} \mathbb{E}\Bigg( \exp\Bigl(-\int_0^1 \dot{\phi}_s^i\, dW^i_s\Bigr) \,\Bigg|\, A^\varepsilon_i \Bigg) = 1.
\end{aligned}
\end{equation*}
Therefore, it remains to analyze the core conditional expectation containing the drift term:
\begin{equation}\label{taylor1}
\mathbb{E}\Bigg( \exp\Bigg( \int_0^1 s^{-\alpha}I_{0^+}^{\alpha}\!\left(s^\alpha\sigma_s^{-1}b_s(\tilde X_s,\tilde Y_s)\right) \cdot dW_s \Bigg) \,\Bigg|\, A^\varepsilon \Bigg).
\end{equation}

We next perform a Taylor expansion on the drift coefficient:
\begin{align*}
b_s(\tilde{X}_s,\tilde{Y}_s) &= b_s(\psi_s,\phi_s) + \nabla_x b_s(\psi_s,\phi_s)(\tilde{X}_s-\psi_s) + \nabla_y b_s(\psi_s,\phi_s)(\tilde{Y}_s-\phi_s) + R_s.
\end{align*}
Substituting this expansion into \eqref{taylor1} reveals that, in contrast to the classical setting, two additional distinct parts emerge. The first is the coupling term originating from the non-degenerate state component, denoted by $\nabla_x b_s(\psi_s,\phi_s)(\tilde{X}_s-\psi_s)$. The second involves cross-component noise coupling terms across different dimensions within $\nabla_y b_s(\psi_s,\phi_s)(\tilde{Y}_s-\phi_s)$, which typically take the form $\int_0^t \sigma^i_s \, dB^{H,i}_s \, dW^j_s$. Fortunately, the GCI-based conditional expectation inequality (Theorem~\ref{conditional-expectation}) provides a powerful and elegant framework to analyze the conditional exponential expectations of both terms.

We only treat the genuinely new terms. First, we consider the state coupling term arising from the non-degeneracy, which leads to the following conditional exponential expectation:
\begin{align*}
&\lim_{\varepsilon\to0} \mathbb{E}\Bigg[ \exp\Bigg( \int_0^1 s^{-\alpha}I_{0^+}^{\alpha} \!\left( s^\alpha\sigma_s^{-1}\nabla_x b_s(\psi_s,\phi_s)(\tilde X_s-\psi_s) \right) \cdot dW_s \Bigg) \,\Bigg|\, A^\varepsilon \Bigg].
\end{align*}
Let $g_s = (g^1_s, \dots, g^d_s)^T$ be the $\mathbb{R}^d$-valued process defined component-wise by
\begin{equation*}
g_s:= s^{-2\alpha}I_{0^+}^{\alpha} \!\left( s^\alpha\sigma_s^{-1}\nabla_x b_s(\psi_s,\phi_s)(\tilde X_s-\psi_s) \right).
\end{equation*}
By the multi-dimensional extension of Lemma~\ref{singular-c1}, each component satisfies the required regularity conditions. Applying the stochastic integration by parts component-wise, we can decompose the dot product integral as
\begin{align*}
\int_0^1 g_s s^\alpha \cdot dW_s &= \sum_{i=1}^d \int_0^1 g^i_s s^\alpha\,dW^i_s \\
&= \sum_{i=1}^d \left[ g^i_1\int_0^1 s^\alpha\,dW^i_s - \int_0^1 (g^i_s)' \left(\int_0^s u^\alpha\,dW^i_u\right)ds \right] =: D_1 + D_2.
\end{align*}

We first estimate $D_1$. Under the conditioning small-ball event, we utilize the coordinate-wise bound $\max_{1 \le i \le d} |g^i_1| \le C\varepsilon$, yielding
\begin{align*}
1 &\le \lim_{\varepsilon\to0} \mathbb{E}\left[ \exp(|D_1|) \,\middle|\, A^\varepsilon \right] \le \lim_{\varepsilon\to0} \mathbb{E}\left[ \exp\left( C\varepsilon \sum_{i=1}^d \left|\int_0^1 s^\alpha\,dW^i_s\right| \right) \,\middle|\, \bigcap_{i=1}^d A^\varepsilon_i \right] = 1,
\end{align*}
where the last equality follows from the decoupling independence and Lemma~\ref{estimate1'}. Hence
 \[\lim_{\varepsilon\to0} \mathbb{E}[ \exp(D_1) | A^\varepsilon ] = 1.\]

For the remainder term $D_2$, since $\max_{1 \le i \le d} \sup_{s \in [0,1]} |(g^i_s)'| \le C\varepsilon$ under the small-ball topology, we obtain the absolute control:
\begin{equation*}
|D_2| \le C\varepsilon \sum_{i=1}^d \int_0^1 \left| \int_0^s u^\alpha\,dW^i_u \right| ds.
\end{equation*}
Exploiting the independence of the components across different dimensions, we can decouple the conditional expectation into a product of one-dimensional expectations:
\begin{align}\label{eq:D2_decouple}
\lim_{\varepsilon\to0} \mathbb{E}\left[ \exp(|D_2|) \,\middle|\, \bigcap_{i=1}^d A^\varepsilon_i \right] \le \lim_{\varepsilon\to0} \prod_{i=1}^d \mathbb{E}\left[ \exp\left( C\varepsilon \int_0^1 \left|\int_0^s u^\alpha\,dW^i_u\right|ds \right) \,\middle|\, A^\varepsilon_i \right].
\end{align}

Now, we focus on the one-dimensional functional inside the expectation. For any two paths $\omega^1, \omega^2$ in the one-dimensional Wiener space and any $\lambda \in (0,1)$, the linearity of the stochastic integral and the triangle inequality yield
\begin{align*}
\int_0^1 \left| \int_0^s u^\alpha \, d(\lambda\omega_u^1+(1-\lambda)\omega_u^2) \right| ds 
&\le \int_0^1 \left( \lambda \left| \int_0^s u^\alpha \, d\omega_u^1 \right| + (1-\lambda) \left| \int_0^s u^\alpha \, d\omega_u^2 \right| \right) ds \\
&= \lambda \int_0^1 \left| \int_0^s u^\alpha \, d\omega_u^1 \right| ds + (1-\lambda) \int_0^1 \left| \int_0^s u^\alpha \, d\omega_u^2 \right| ds.
\end{align*}
Since the exponential function $x \mapsto \exp(x)$ is convex and monotonically increasing, applying it to the above inequality gives
\begin{align*}
& \exp\left( C\varepsilon \int_0^1 \left| \int_0^s u^\alpha \, d(\lambda\omega_u^1+(1-\lambda)\omega_u^2) \right| ds \right) \\
&\le \exp\left( \lambda C\varepsilon \int_0^1 \left| \int_0^s u^\alpha \, d\omega_u^1 \right| ds + (1-\lambda) C\varepsilon \int_0^1 \left| \int_0^s u^\alpha \, d\omega_u^2 \right| ds \right) \\
&\le \lambda \exp\left( C\varepsilon \int_0^1 \left| \int_0^s u^\alpha \, d\omega_u^1 \right| ds \right) + (1-\lambda) \exp\left( C\varepsilon \int_0^1 \left| \int_0^s u^\alpha \, d\omega_u^2 \right| ds \right).
\end{align*}
This explicit bound, along with the trivial symmetry property under the transformation $\omega \mapsto -\omega$, confirms that for each dimension $i$, the mapping
\begin{equation*}
\omega \mapsto \exp\left( C\varepsilon \int_0^1 \left|\int_0^s u^\alpha\,d\omega_u\right|ds \right)
\end{equation*}
defines a non-negative, convex, and symmetric functional on the one-dimensional Wiener space. 

By using Theorem~\ref{conditional-expectation} individually for each component, we can drop the conditioning to obtain the upper bound:
\begin{align*}
\lim_{\varepsilon\to0} \prod_{i=1}^d \mathbb{E}\left[ \exp\left( C\varepsilon \int_0^1 \left|\int_0^s u^\alpha\,dW^i_u\right|ds \right) \,\middle|\, A^\varepsilon_i \right] 
&\le \prod_{i=1}^d \lim_{\varepsilon\to0} \mathbb{E}\left[ \exp\left( C\varepsilon \int_0^1 \left|\int_0^s u^\alpha\,dW^i_u\right|ds \right) \right] = 1,
\end{align*}
where the final convergence to $1$ for each component is strictly guaranteed by Fernique's theorem. Consequently, we conclude that the conditional expectation for $D_2$ converges to $1$. Combining the results for $D_1$ and $D_2$ completes the proof of the state coupling term.

Next, we investigate the coupling terms arising from the multi-dimensional noise. We consider the conditional exponential expectation given by
\begin{equation*}
\mathbb{E}\Bigg( \exp\Bigg( \int_0^1 s^{-\alpha}I_{0^+}^{\alpha}\!\left(s^\alpha\sigma_s^{-1}\nabla_y b_s(\psi_s,\phi_s)(\tilde{Y}_s-\phi_s)\right) \cdot dW_s \Bigg) \,\Bigg|\, A^\varepsilon \Bigg).
\end{equation*}
To proceed, we decompose the inner product into its coordinate components:
\begin{equation*}
\sum_{i,j=1}^d \int_0^1 s^{-\alpha}I_{0^+}^{\alpha}\!\left(s^\alpha(\sigma^i_s)^{-1} \partial_{y_j} b^i_s(\psi_s,\phi_s)\int_0^s\sigma_u^j \, dB^{H,j}_u \right) dW^i_s.
\end{equation*}

When $i = j$, due to the independence of the noise components across different dimensions, the conditional expectation simplifies to:
\begin{align*}
    & \mathbb{E}\left( \exp\left( \int_0^1 s^{-\alpha}I_{0^+}^{\alpha}\!\left(s^\alpha(\sigma^i_s)^{-1}  \partial_{y_i} b^i_s(\psi_s,\phi_s)\int_0^s\sigma_u^i \, dB^{H,i}_u \right) dW^i_s \right) \;\Bigg|\; \bigcap_{k=1}^d  A_k ^\varepsilon\right) \\
    &= \mathbb{E}\left( \exp\left( \int_0^1 s^{-\alpha}I_{0^+}^{\alpha}\!\left(s^\alpha(\sigma^i_s)^{-1}  \partial_{y_i} b^i_s(\psi_s,\phi_s)\int_0^s\sigma_u^i \, dB^{H,i}_u \right) dW^i_s  \right) \;\Bigg|\; A_i^\varepsilon \right).
\end{align*} 
According to Lemma~\ref{estimate2}, this diagonal portion transforms into a double stochastic integral corresponding to the trace of the integrand. Summing over all $i=j$, we obtain the classical divergence term:
\begin{align*}
     &\lim_{\varepsilon\to 0} \mathbb{E}\left( \exp\left(\sum_{i=1}^d \int_0^1 s^{-\alpha}I_{0^+}^{\alpha}\!\left(s^\alpha(\sigma^i_s)^{-1}  \partial_{y_i} b^i_s(\psi_s,\phi_s)\int_0^s\sigma_u^i \, dB^{H,i}_u \right) dW^i_s \right) \;\Bigg|\; A^\varepsilon \right) \\& = \exp\left(-\frac{d_H}{2} \int_0^1 \nabla_y \cdot b_s(\psi_s,\phi_s) \, ds \right).
\end{align*}

It remains to treat the additional cross-coupling terms where $i \neq j$. By independence, the relevant conditional expectation isolates to the events $A_i$ and $A_j$:
\begin{align*}
      & \mathbb{E}\left( \exp\left( \int_0^1 s^{-\alpha}I_{0^+}^{\alpha}\!\left(s^\alpha(\sigma^i_s)^{-1}  \partial_{y_j} b^i_s(\psi_s,\phi_s)\int_0^s\sigma_u^j \, dB^{H,j}_u \right) dW^i_s\right) \;\Bigg|\; A^\varepsilon \right) \\
    &= \mathbb{E}\left( \exp\left( \int_0^1 s^{-\alpha}I_{0^+}^{\alpha}\!\left(s^\alpha(\sigma^i_s)^{-1}  \partial_{y_j} b^i_s(\psi_s,\phi_s)\int_0^s\sigma_u^j \, dB^{H,j}_u \right) dW^i_s \right) \;\Bigg|\; A_i^\varepsilon, A_j^\varepsilon \right).
\end{align*}

Let us denote the predictable integrand, which solely depends on the path $\omega_j$, as 
\begin{equation*}
    h_s(\omega_j) := s^{-\alpha} I^\alpha_{0^+} \!\left( s^\alpha (\sigma_s^{i})^{-1}\partial_{y_j} b_s^i(\psi_s,\phi_s) \int_0^s \sigma_u^j\,dB^{H,j}_u(\omega_j) \right).
\end{equation*}
By Fubini's theorem (via the tower property of conditional expectation) and the independence of the product measure $\mathbb{P} = \mathbb{P}_i \otimes \mathbb{P}_j$, we can rewrite the expectation as an iterated integral:
\begin{align*}
     &\mathbb{E}\left( \exp\left( \int_0^1 h_s(\omega_j)\, dW^i_s  \right) \;\Bigg|\; A_i^\varepsilon, A_j^\varepsilon \right)\\
    &= \int_{\Omega_j} \left( \int_{\Omega_i} \exp\left( \int_0^1 h_s(\omega_j)\, dW^i_s(\omega_i)\right) \frac{\mathbf{1}_{A_i^\varepsilon}(\omega_i)}{\mathbb{P}(A_i^\varepsilon)} d\mathbb{P}_i(\omega_i) \right) \frac{\mathbf{1}_{A_j^\varepsilon}(\omega_j)}{\mathbb{P}(A_j^\varepsilon)} d\mathbb{P}_j(\omega_j) \\
    &= \mathbb{E}_j\left( \mathbb{E}_i\left(\exp\left( \int_0^1 h_s(\omega_j)\, dW^i_s (\omega_i) \right) \Bigg|\; A_i^\varepsilon \right) \;\Bigg|\; A_j^\varepsilon \right).
\end{align*}
To rigorously evaluate the limit of this expression as $\varepsilon \to 0$, we establish uniform bounds for the inner conditional expectation. First, for the lower bound, we apply Jensen's inequality. For a fixed $\omega_j$, the stochastic integral $\int_0^1 h_s(\omega_j) dW^i_s$ is a centered Gaussian random variable (which is an odd function of $\omega_i$). Since the conditioning set $A_i^\varepsilon$ is a symmetric convex set centered at the origin, the conditional expectation of this odd functional over $A_i^\varepsilon$ is exactly zero. Thus,
\begin{align*}
    \mathbb{E}_i\left(\exp\left( \int_0^1 h_s(\omega_j)\, dW^i_s(\omega_i) \right) \Bigg|\; A_i^\varepsilon \right) 
    &\geq \exp\left( \mathbb{E}_i\left( \int_0^1 h_s(\omega_j)\, dW^i_s(\omega_i) \Bigg|\; A_i ^\varepsilon \right) \right) = \exp(0) = 1.
\end{align*}

For the upper bound, we leverage Theorem~\ref{conditional-expectation}. The exponential function can be symmetrized to $\cosh(x)$, which is an even and convex function. Conditioning a symmetric Gaussian measure on a symmetric convex set $A_i^\varepsilon$ strictly decreases the expectation of an even convex function. Hence, the conditional expectation is bounded above by the unconditional expectation:
\begin{align*}
    \mathbb{E}_i\left(\exp\left( \int_0^1 h_s(\omega_j)\, dW^i_s(\omega_i) \right) \Bigg|\; A_i^\varepsilon \right) 
    &\le \mathbb{E}_i\left(\exp\left( \int_0^1 h_s(\omega_j)\, dW^i_s(\omega_i) \right) \right) \\
    &= \exp\left( \frac{1}{2} \int_0^1 |h_s(\omega_j)|^2 \, ds \right).
\end{align*}
Crucially, on the outer conditioning event $A_j^\varepsilon$, the path $\omega_j$ satisfies the small-ball condition. By Lemma~\ref{singular-c1}, this implies the uniform bound $\sup_{s \in [0,1]} |h_s(\omega_j)| \leq C\varepsilon$ for some constant $C>0$. Therefore, the upper bound is uniformly controlled by $\exp(C^2 \varepsilon^2 / 2)$.

Consequently, for almost every $\omega_j \in A_j^\varepsilon$, the inner expectation is uniformly squeezed:
\begin{equation*}
    1 \leq \mathbb{E}_i\left(\exp\left( \int_0^1 h_s(\omega_j)\, dW^i_s(\omega_i) \right) \Bigg|\; A_i^\varepsilon \right) \leq \exp\left(\frac{C^2 \varepsilon^2}{2}\right).
\end{equation*}
Taking the outer conditional expectation $\mathbb{E}_j(\,\cdot \mid A_j^\varepsilon)$ preserves these constant bounds. By applying the Squeeze Theorem as $\varepsilon \to 0$, the measure dependency completely vanishes, and we conclude that the cross-coupling terms contribute trivially:
\begin{equation*}
    \lim_{\varepsilon\to 0} \mathbb{E}\left( \exp\left( \int_0^1 h_s(\omega_j)\, dW^i_s  \right) \;\Bigg|\; A_i^\varepsilon, A_j^\varepsilon \right) = 1.
\end{equation*}

Finally, all remaining terms can be handled exactly as in \cite{Nualart}. Therefore, we obtain the functional form for Case I:
\begin{equation*}
J(\phi) = \frac12\int_0^1 \left| \dot{\phi}_s - s^{-\alpha}I_{0^+}^{\alpha}\!\left(s^\alpha\sigma_s^{-1}b_s(\psi_s,\phi_s)\right) \right|^2 ds + \frac{d_H}{2} \int_0^1 \nabla_y \cdot b_s(\psi_s,\phi_s)\,ds.
\end{equation*}

\medskip

\noindent\textbf{Case II: $1/2<H<1$.}

In this regime, we aim to prove that the OM functional is given by
\begin{equation*}
J(\phi) = \frac12\int_0^1 \left| \dot{\phi}_s - s^{\alpha}D_{0^+}^{\alpha}\!\left( s^{-\alpha}\sigma_s^{-1}b_s(\psi_s,\phi_s) \right) \right|^2 ds + \frac{d_H}{2} \int_0^1 \nabla_y\cdot b_s(\psi_s,\phi_s)\,ds.
\end{equation*}
As in Case I, all arguments identical to the standard non-degenerate setting are omitted. We focus directly on the new term generated by the degenerate noise structure. It therefore remains to analyze the conditional exponential expectation:
\begin{equation*}
\mathbb{E}\Bigg[ \exp\Bigg( \int_0^1 s^{\alpha}D_{0^+}^{\alpha}\!\left( s^{-\alpha}\sigma_s^{-1}\nabla_x b_s(\psi_s,\phi_s)(\tilde X_s-\psi_s) \right)\cdot dW_s \Bigg) \,\Bigg|\, A^\varepsilon \Bigg].
\end{equation*}
Let
\begin{equation*}
g_s= s^{2\alpha}D_{0^+}^{\alpha}\!\left( s^{-\alpha}\sigma_s^{-1}\nabla_x b_s(\psi_s,\phi_s)(\tilde X_s-\psi_s) \right).
\end{equation*}
By applying the multi-dimensional extension of Lemma~\ref{regular-c1}, we can perform stochastic integration by parts on the inner product integral to yield:
\begin{align*}
\int_0^1 s^{\alpha}D_{0^+}^{\alpha}\!\left( s^{-\alpha}\sigma_s^{-1}\nabla_x b_s(\psi_s,\phi_s)(\tilde X_s-\psi_s) \right)\cdot dW_s
&= \int_0^1 s^{-\alpha} g_s\cdot dW_s \\
&= g_1\cdot \int_0^1 s^{-\alpha}\,dW_s - \int_0^1 g_s' \cdot \left(\int_0^s u^{-\alpha}\,dW_u\right)ds \\
&=:D_1+D_2.
\end{align*}
Proceeding exactly as in the singular case, and applying Lemma~\ref{estimate1'}, Theorem~\ref{conditional-expectation}, and Lemma~\ref{regular-c1}, we obtain
\begin{equation*}
\lim_{\varepsilon\to0} \mathbb{E}\Bigg[ \exp\Bigg( \int_0^1 s^{\alpha}D_{0^+}^{\alpha}\!\left( s^{-\alpha}\sigma_s^{-1}\nabla_x b_s(\psi_s,\phi_s)(\tilde X_s-\psi_s) \right)dW_s \Bigg) \,\Bigg|\, A^\varepsilon \Bigg] = 1.
\end{equation*}
For the cross-coupling terms of the multi-dimensional noise, substituting fractional integrals with fractional derivatives does not introduce any qualitative deviations in the squeezing argument. Therefore, we can directly conclude that the OM functional for Case II takes the form:
\begin{equation*}
J(\phi) = \frac12\int_0^1 \left| \dot{\phi}_s - s^{\alpha}D_{0^+}^{\alpha}\!\left( s^{-\alpha}\sigma_s^{-1}b_s(\psi_s,\phi_s) \right) \right|^2 ds + \frac{d_H}{2} \int_0^1 \nabla_y \cdot b_s(\psi_s,\phi_s)\,ds.
\end{equation*}

\medskip

\noindent\textbf{Case III: $H=1/2$.}

In the classical Brownian case ($H=1/2$), the proof becomes substantially simpler because both the small-ball constraint topology and the driving noise appearing in the expectation are standard Brownian motions. We directly provide the resulting classical Lagrangian expression as:
\begin{equation*}
J(\phi) = \frac12\int_0^1 \left| \sigma_s^{-1} \left( \dot{\phi}_s - b_s(\psi_s,\phi_s) \right) \right|^2 ds + \frac{1}{2}\int_0^1 \nabla_y\cdot b_s(\psi_s,\phi_s)\,ds. 
\end{equation*}
The proof is now complete.
\end{proof}

\begin{remark}
    To the best of our knowledge, recent LDP results for fractional noise \cite{FanYuYuan2023} demonstrate that the rate function corresponds to the kinetic energy term of the OM functional. Thus, we anticipate that for the degenerate noise case, the associated LDP rate function can similarly be proved to take the form:
    \begin{equation}\label{rate_fn}
        I(\phi)= \frac{1}{2}\int_0^1 \left| (K_H^\sigma)^{-1} \left( \phi_\cdot - y_0 - \int_0^\cdot b_v(\psi_v, \phi_v) \, dv \right)(s) \right|^2 ds,
    \end{equation}
    for $\phi-y_0\in \mathcal{H}_H$, where $\psi$ and $\phi$ are coupled via \eqref{sec}.
\end{remark}

\section{Euler--Lagrange Equations, Path Preservation and Numerical Experiments}

In this section, we investigate the properties of the most probable transition paths and the most probable paths by utilizing the explicit expression of the OM functional derived in \eqref{omexpress}. In the first subsection, we deduce the Euler--Lagrange equations, which establish the necessary optimality conditions that the most probable transition paths must satisfy. In the second subsection, we establish a divergence condition on the drift term for an SDE to preserve the most probable path. In the third subsection, we conduct numerical simulations on specific governing equations to empirically validate the primary theoretical results of this paper.

\subsection{Euler--Lagrange Equations}

In this subsection, we derive the Euler--Lagrange equations associated with
the OM functional \eqref{mainresultfunctional} in order to characterize the
optimality conditions satisfied by the most probable transition paths. The
original action functional is given by
\begin{equation*}
    J(\phi)
    =
    \int_0^1 L(s,\psi_s,\phi_s)\,ds,
\end{equation*}
where the state trajectory $\psi$ is implicitly determined by the reference
trajectory $\phi$ through the coupled differential constraint \eqref{sec}.
To systematically handle this differential constraint, we introduce a
time-dependent Lagrange multiplier $\lambda_s$, equivalently the adjoint
variable, and define the augmented Lagrangian
\begin{equation}
    \mathcal{L}(s,\psi_s,\phi_s,\lambda_s)
    =
    L(s,\psi_s,\phi_s)
    +
    \lambda_s^T\bigl(\psi'_s-a_s(\psi_s,\phi_s)\bigr),
\end{equation}
which induces the augmented action functional
\begin{equation}
    \mathcal{J}(\psi,\phi,\lambda)
    =
    \int_0^1
    \mathcal{L}(s,\psi_s,\phi_s,\lambda_s)\,ds.
\end{equation}

Accordingly, any smooth minimizer of the constrained variational problem must
be a stationary point of $\mathcal{J}$; equivalently, the first-order variations
with respect to the triplet $(\psi,\phi,\lambda)$ vanish simultaneously.

We first consider the reference case $H=\frac12$. Introducing the
positive-definite matrix
\[
    \Sigma_s^{-1}:=(\sigma_s\sigma_s^T)^{-1},
\]
and omitting irrelevant normalization constants, the augmented action
functional can be written in the form
\begin{equation*}
\begin{aligned}
\mathcal{J}(\psi,\phi,\lambda)
=
\int_0^1 \Big[
&
\bigl(\phi'_s-b_s(\psi_s,\phi_s)\bigr)^T
\Sigma_s^{-1}
\bigl(\phi'_s-b_s(\psi_s,\phi_s)\bigr)
\\
&
+
\nabla_y\cdot b_s(\psi_s,\phi_s)
+
\lambda_s^T
\bigl(\psi'_s-a_s(\psi_s,\phi_s)\bigr)
\Big]\,ds.
\end{aligned}
\end{equation*}

To compute the first variations, we introduce independent perturbation fields
with suitable dimensions and boundary conditions compatible with the
variational structure of the system. More precisely, let
\[
    \xi\in\mathcal{C}_\psi
    :=
    \bigl\{
    \xi\in C^\infty([0,1];\mathbb{R}^m)
    :
    \xi_0=0
    \bigr\},
\]
together with
\[
    \zeta\in C_0^\infty([0,1];\mathbb{R}^m),
    \qquad
    \eta\in C_0^\infty([0,1];\mathbb{R}^n).
\]

The boundary conditions imposed on these perturbations are dictated by the
differential order of the resulting Euler--Lagrange system. Specifically, both
the state variable $\psi$ and the adjoint variable $\lambda$ satisfy first-order
differential equations, whereas the reference path $\phi$ satisfies a
second-order equation. Since the initial configurations $(\psi_0,\phi_0)$ are
fixed by the underlying stochastic dynamics, the remaining boundary degrees of
freedom are naturally assigned to the terminal condition of $\lambda$ and the
endpoint constraints of $\phi$.

We first compute the variation with respect to the state variable $\psi$ in the
direction $\xi\in\mathcal{C}_\psi$. Since $\xi_0=0$, the first variation is
given by
\begin{equation}\label{var_psi}
\begin{aligned}
\delta_\psi\mathcal{J}
=
\int_0^1 \Big[
&
-2
\bigl(
\phi'_s-b_s(\psi_s,\phi_s)
\bigr)^T
\Sigma_s^{-1}
\nabla_x b_s(\psi_s,\phi_s)
\\
&
+
\nabla_x
\bigl(
\nabla_y\cdot b_s(\psi_s,\phi_s)
\bigr)
-
\lambda_s'^T
-
\lambda_s^T
\nabla_x a_s(\psi_s,\phi_s)
\Big]
\xi_s\,ds.
\end{aligned}
\end{equation}
Here we have used the integration-by-parts identity
\begin{equation*}
    \int_0^1
    \lambda_s^T\xi'_s\,ds
    =
    \lambda_1^T\xi_1
    -
    \lambda_0^T\xi_0
    -
    \int_0^1
    \lambda_s'^T\xi_s\,ds.
\end{equation*}
Since $\xi_0=0$ while the terminal perturbation $\xi_1$ remains arbitrary,
cancellation of the boundary contribution requires the natural transversality
condition
\[
    \lambda_1=0.
\]

Next, varying the multiplier $\lambda$ in the direction
$\zeta\in C_0^\infty([0,1];\mathbb{R}^m)$ immediately recovers the forward
state equation:
\begin{equation}\label{var_lambda}
    \delta_\lambda\mathcal{J}
    =
    \int_0^1
    \zeta_s^T
    \bigl(
    \psi'_s-a_s(\psi_s,\phi_s)
    \bigr)\,ds.
\end{equation}

Finally, we perturb the reference trajectory $\phi$ along
$\eta\in C_0^\infty([0,1];\mathbb{R}^n)$ satisfying $\eta_0=\eta_1=0$.
After integration by parts with respect to the velocity perturbation $\eta'_s$,
we obtain
\begin{equation}\label{var_phi}
\begin{aligned}
\delta_\phi\mathcal{J}
=
\int_0^1 \Big[
&
-2
\frac{d}{ds}
\left(
\bigl(
\phi'_s-b_s(\psi_s,\phi_s)
\bigr)^T
\Sigma_s^{-1}
\right)
\\
&
-
2
\bigl(
\phi'_s-b_s(\psi_s,\phi_s)
\bigr)^T
\Sigma_s^{-1}
\nabla_y b_s(\psi_s,\phi_s)
\\
&
+
\nabla_y
\bigl(
\nabla_y\cdot b_s(\psi_s,\phi_s)
\bigr)
-
\lambda_s^T
\nabla_y a_s(\psi_s,\phi_s)
\Big]
\eta_s\,ds.
\end{aligned}
\end{equation}

Applying the Fundamental Lemma of the Calculus of Variations and taking
transposes of the variational terms, we arrive at the fully expanded coupled
Euler--Lagrange system governing the multidimensional most probable transition
paths:
\begin{subequations}\label{EL_explicit_system}
\begin{numcases}{}
    \psi'_s
    =
    a_s(\psi_s,\phi_s),
    \label{EL_explicit_state}
    \\[12pt]
    \begin{aligned}
    \lambda'_s
    =
    &
    -
    \bigl[
    \nabla_x a_s(\psi_s,\phi_s)
    \bigr]^T
    \lambda_s
    \\
    &
    -
    2
    \bigl[
    \nabla_x b_s(\psi_s,\phi_s)
    \bigr]^T
    \Sigma_s^{-1}
    \bigl(
    \phi'_s-b_s(\psi_s,\phi_s)
    \bigr)
    \\
    &
    +
    \bigl[
    \nabla_x
    (
    \nabla_y\cdot b_s(\psi_s,\phi_s)
    )
    \bigr]^T,
    \end{aligned}
    \label{EL_explicit_adjoint}
    \\[12pt]
    \begin{aligned}
    2
    \frac{d}{ds}
    \left[
    \Sigma_s^{-1}
    \bigl(
    \phi'_s-b_s(\psi_s,\phi_s)
    \bigr)
    \right]
    =
    &
    -
    2
    \bigl[
    \nabla_y b_s(\psi_s,\phi_s)
    \bigr]^T
    \Sigma_s^{-1}
    \bigl(
    \phi'_s-b_s(\psi_s,\phi_s)
    \bigr)
    \\
    &
    +
    \bigl[
    \nabla_y
    (
    \nabla_y\cdot b_s(\psi_s,\phi_s)
    )
    \bigr]^T
    \\
    &
    -
    \bigl[
    \nabla_y a_s(\psi_s,\phi_s)
    \bigr]^T
    \lambda_s.
    \end{aligned}
    \label{EL_explicit_optimality}
\end{numcases}
\end{subequations}
subject to the two-point boundary conditions
\begin{equation}\label{boundary condition}
\begin{cases}
    \psi_0=x_0\in\mathbb{R}^m,
    & \text{(initial condition),}
    \\[4pt]
    \phi_0=y_0,
    \quad
    \phi_1=y_1\in\mathbb{R}^n,
    & \text{(boundary conditions),}
    \\[4pt]
    \lambda_1=0\in\mathbb{R}^m,
    & \text{(transversality condition).}
\end{cases}
\end{equation}

In this formulation, \eqref{EL_explicit_state} describes the forward
propagation of the state variable in $\mathbb{R}^m$,
\eqref{EL_explicit_adjoint} governs the backward evolution of the adjoint
variable $\lambda_s$, and \eqref{EL_explicit_optimality} yields the second-order
optimality condition satisfied by the critical trajectory $\phi_s$ in
$\mathbb{R}^n$.

We now turn to the fractional setting. In this case, the augmented action
functional takes the form
\begin{equation*}
\begin{aligned}
\mathcal{J}(\psi,\phi,\lambda)
=
\int_0^1 \Bigg[
&
\left|
(K_H^\sigma)^{-1}
\left(
\phi_\cdot-y_0-\int_0^\cdot
b_u(\psi_u,\phi_u)\,du
\right)(s)
\right|^2
\\
&
+
d_H
\nabla_y\cdot b_s(\psi_s,\phi_s)
+
\lambda_s^{T}
\bigl(
\psi'_s-a_s(\psi_s,\phi_s)
\bigr)
\Bigg]\,ds.
\end{aligned}
\end{equation*}

For notational convenience, define
\[
    G_s
    :=
    \Biggl(
    \mathcal{K}_H^\sigma
    \circ
    (K_H^\sigma)^{-1}
    \left[
    \phi_\cdot-y_0
    -
    \int_0^\cdot
    b_u(\psi_u,\phi_u)\,du
    \right]
    \Biggr)(s).
\]
By applying the fractional integration-by-parts formula \eqref{fffubini}, the
left-sided fractional operators are transferred to their corresponding
right-sided counterparts. Performing the first variation then leads to the
following fractional Euler--Lagrange system:
\begin{subequations}\label{EL_fractional_system}
\begin{numcases}{}
\psi'_s = a_s(\psi_s,\phi_s),
\label{EL_frac_state}
\\[12pt]
\begin{aligned}
\lambda'_s
= &
-
\bigl[
\nabla_x a_s(\psi_s,\phi_s)
\bigr]^T
\lambda_s
\\
&
-
2
\bigl[
\nabla_x b_s(\psi_s,\phi_s)
\bigr]^T
G_s
\\
&
+
d_H
\bigl[
\nabla_x
(
\nabla_y\cdot b_s(\psi_s,\phi_s)
)
\bigr]^T,
\end{aligned}
\label{EL_frac_adjoint}
\\[12pt]
\begin{aligned}
2
\left(
\frac{d}{ds}
+
\bigl[
\nabla_y b_s(\psi_s,\phi_s)
\bigr]^T
\right)
G_s
= &
d_H
\bigl[
\nabla_y
(
\nabla_y\cdot b_s(\psi_s,\phi_s)
)
\bigr]^T
\\
&
-
\bigl[
\nabla_y a_s(\psi_s,\phi_s)
\bigr]^T
\lambda_s.
\end{aligned}
\label{EL_frac_optimality}
\end{numcases}
\end{subequations}
Here the operator $\mathcal{K}_H^\sigma$ is defined by
\[
\mathcal{K}_H^\sigma f
=
\begin{cases}
(\sigma_s^{-1})^T
\,s^{-\alpha}
D_{1^-}^{\alpha}
\bigl(
s^{\alpha}f_s
\bigr),
& \frac12<H<1,
\\[2mm]
(\sigma_s^{-1})^T
\,s^{\alpha}
I_{1^-}^{\alpha}
\bigl(
s^{-\alpha}f_s
\bigr),
& \frac14<H<\frac12.
\end{cases}
\]
The associated boundary conditions are identical to those in
\eqref{boundary condition}.

\begin{remark}
By the fractional Sobolev embedding, the Cameron--Martin space
$\mathcal{H}_H^\sigma([0,1];\mathbb{R}^n)$ is continuously embedded into
$C^\gamma_0([0,1];\mathbb{R}^n)$ for every $\gamma<H$. Hence elements of
$\mathcal{H}_H^\sigma([0,1];\mathbb{R}^n)$ admit continuous representatives,
and it is therefore meaningful to impose endpoint conditions such as
$\phi_0=y_0$ and $\phi_1=y_1$.
\end{remark}

\begin{remark} \label{rem:boundary_transfer}
As evidenced by the variational structure, for a general coupled system, we are typically restricted to investigating the most probable transition of the control component $\phi$ from an initial state $\phi_0 = y_0$ to a terminal configuration $\phi_1 = y_1$. However, if the underlying system \eqref{fir} possesses a specific mechanical structure (for instance, the classical kinematic constraint $\psi'_s = \phi_s$), the terminal boundary degrees of freedom absorbed by the adjoint variable $\lambda$ can be systematically transferred to the state variable $\psi$. In such specialized scenarios, the system condenses into a higher-order differential equation, allowing us to explicitly prescribe and investigate the full state-control transition for the pair $(\psi, \phi)$ from $(x_0, y_0)$ to $(x_1, y_1)$.
\end{remark}

	\subsection{Most Probable Path Preservation Theorem}

	The Euler--Lagrange equations derived above characterize the most probable transition paths connecting prescribed endpoints and therefore depend on both the initial and terminal configurations. It is natural to ask whether one can identify distinguished trajectories when the terminal constraint is removed.

To this end, we consider a free-terminal formulation of the most probable path problem. A path $\phi^*$ satisfying $\phi^*-y_0\in\mathcal H^\sigma_H$ is called a free-terminal most probable path of system \eqref{fir} if it minimizes the OM functional over the admissible class, namely,
\[
J(\phi^*)
=
\inf_{\phi-y_0\in\mathcal H^\sigma_H} J(\phi).
\]

We say that a stochastic system possesses the path-preservation property if every solution trajectory of the corresponding deterministic (noise-free) system is also a free-terminal most probable path of the stochastic system.

Path-preservation phenomena have previously been observed for certain classes of stochastic Hamiltonian systems. In particular, \cite{xinze2026most} showed that, under suitable structural assumptions, the introduction of non-degenerate stochastic perturbations does not alter the set of most probable paths. The result established below demonstrates that an analogous property continues to hold for the degenerate systems considered in the present work, although the underlying mechanism is fundamentally different.

As will be seen from the explicit expression of the OM functional \eqref{omexpress}, the action consists of a nonnegative kinetic term and a divergence contribution associated with the deterministic flow. Under an appropriate constant-divergence condition, the latter becomes path-independent and therefore does not influence the minimization problem. This observation leads to the following path-preservation theorem.

In what follows, we investigate the path-preservation property for the coupled degenerate structure considered in this work. As observed from the explicit expression of the Onsager--Machlup functional \eqref{omexpress}, the energy dissipation of a path consists of two distinct components: the kinetic energy term, and the divergence term $\nabla_y \cdot b(\psi, \phi)$ which characterizes the phase-space volume evolution of the non-degenerate $y$-component under the deterministic flow. Consequently, if we require the contraction or expansion rate of the phase volume in the $y$-subspace to be spatially uniform, the non-local influence of this divergence contribution vanishes globally. This mechanism leads to the following theorem.

\begin{theorem} \label{mpp}
    Consider the degenerate stochastic system \eqref{fir}. If the drift coefficient of the non-degenerate component satisfies 
    \begin{equation*}
        \nabla_y \cdot b(x,y) = C, \quad \forall (x,y) \in \mathbb{R}^m \times \mathbb{R}^n,
    \end{equation*}
    where $C$ is a constant, then the system possesses the path-preservation property.
\end{theorem}

\begin{proof}
    Let $(\psi^*_t, \phi^*_t)_{t \in [0,1]}$ denote the unique solution trajectory of the deterministic system corresponding to \eqref{fir} issued from the initial condition $(x_0, y_0)$. According to the explicit formulation of the OM functional \eqref{omexpress}, for any valid reference path $\phi$, the functional reads
    \begin{equation*}
        J(\phi) = \frac{1}{2} \int_0^1 \left| (K_H^\sigma)^{-1} \left( \phi_{\cdot} - y_0 - \int_0^\cdot b_v(\psi_v, \phi_v) \, dv \right)(s) \right|^2 ds +  \frac{Cd_H}{2}   .
    \end{equation*}
    Since the deterministic trajectory $(\psi^*, \phi^*)$ satisfies the noiseless governing equations, the terms inside the fractional inverse operator cancel out exactly, which yields
    \begin{equation*}
        \frac{Cd_H}{2} = J(\phi^*) = \inf_{\phi - y_0 \in \mathcal{H}^\sigma_H} J(\phi).
    \end{equation*}
    The proof is complete.
\end{proof}

As a direct application to stochastic Hamiltonian dynamics, we immediately obtain the following corollary.

\begin{corollary}
    Consider the stochastic Hamiltonian system driven by fractional Brownian motion:
    \begin{equation*}
        \begin{cases}
            dq_t = \nabla_p H(q_t,p_t)dt, \\[4pt]
            dp_t = -\nabla_q H(q_t,p_t)dt + \sigma_t dB^H_t.
        \end{cases}
    \end{equation*}
    If the Hamiltonian function $H(q,p)$ satisfies
    \begin{equation*}
        \nabla_p \cdot \bigl(\nabla_q H(q,p)\bigr) = C,
    \end{equation*}
    for some constant $C$, then the stochastic Hamiltonian system possesses the path-preservation property.
\end{corollary}

\subsection{Numerical Experiments}

In this subsection, we conduct numerical experiments to empirically validate the theoretical framework established in the preceding sections. To provide concrete insights, we demonstrate the applicability of our primary findings through two representative examples drawn from classical mechanical systems. 

    The first example is the pendulum equation, which is employed to demonstrate the persistence of trajectories in the sense of the most probable path.

\begin{example}[The Pendulum Equation]
Consider the second-order stochastic differential equation modeling a driven pendulum:
\begin{equation}\label{example}
    X_t'' = -\gamma X_t' - k \sin(X_t) + \big(\sigma_0 + A\cos(\omega t)\big)\,\xi_t^H,
\end{equation}
subject to the initial condition
\[
    \big(X_0, X'_0\big) = \big(-\tfrac{\pi}{2}, 0\big),
\]
where the constant $k$ is given by
\[
    k = \frac{1}{2}\left( \frac{\sqrt{\pi}\,\Gamma(1/4)}{\Gamma(3/4)} \right)^{\!2}.
\] 

For simplicity, we focus on the undamped case by setting $\gamma = 0$. According to Theorem~\ref{mpp}, equation~\eqref{example} preserves the most probable path under these conditions.

Numerical simulations are presented in Figs.~\ref{fig1a}, \ref{fig1b}, and \ref{fig1c}. Furthermore, for the regular regime case shown in Fig.~\ref{fig1c}, the theoretical regularity conditions are fully satisfied. As observed in the figures, the most probable path is exceptionally close to the mean path, with both situated at the center of the probability cloud formed by the sample trajectories. Notably, although the noise frequency and the Hurst index vary, the most probable path remains completely unaffected.
\end{example}
\begin{figure}[htbp]
    \centering
    % 第一张图：H=0.3
    \includegraphics[width=0.48\textwidth]{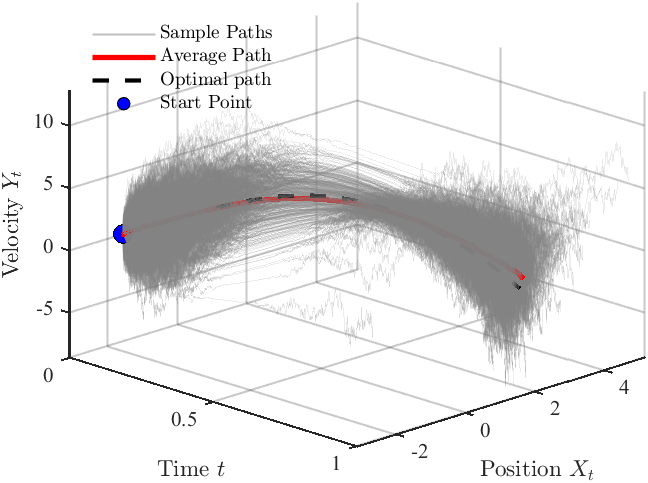}
   \caption{Simulation of the most probable path for $H=0.3$, $\sigma_0=2$, $A=1.5$, and $\omega=2$.}
    \label{fig1a}
    
    \vspace{0.3cm} % 每张图之间留出适当的物理空隙，防止上下标题粘连
    
    % 第二张图：H=0.5
    \includegraphics[width=0.48\textwidth]{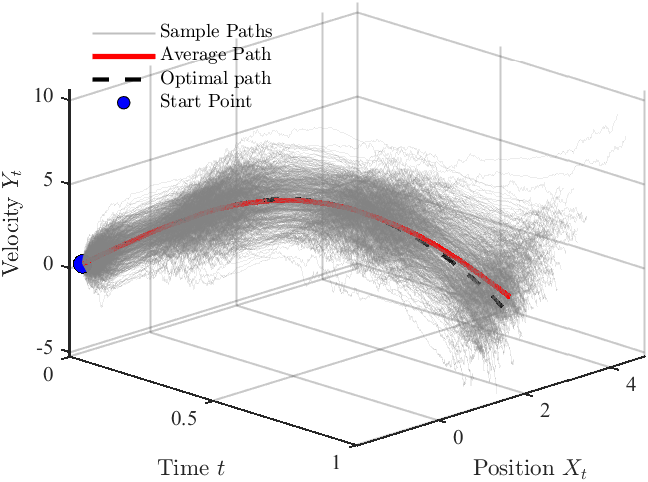}
    \caption{Simulation of the most probable path for $H=0.5, \sigma_0=2, A=1.5, \omega=6\pi$.}
    \label{fig1b}
    
    \vspace{0.3cm}
    
    % 第三张图：H=0.8
    \includegraphics[width=0.48\textwidth]{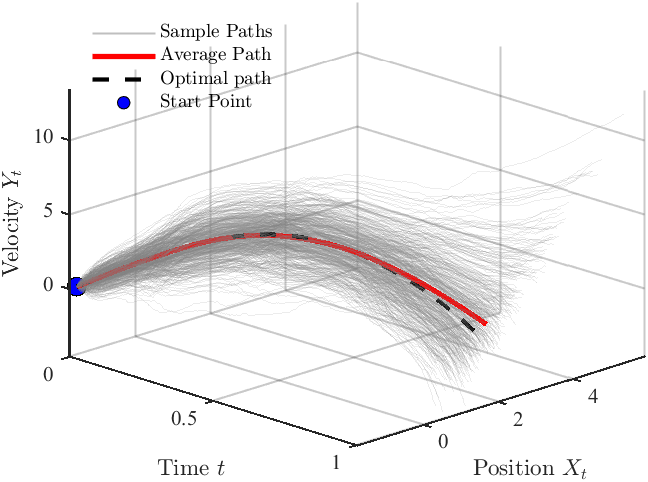}
    \caption{Simulation of the most probable path for $H=0.8, \sigma_0=3, A=1.5, \omega=8\pi$.}
    \label{fig1c}
\end{figure}

The second example concerns a stochastic Duffing equation, a quintessential model for nonlinear oscillations. Here, we investigate the most probable transition path between two stable equilibrium points.

\begin{example}[The Duffing Equation]
Consider the following system:
\begin{equation} \label{equation 2}
    X''_t + \gamma X_t' + V'(X_t) = \sigma_t \xi_t^H.
\end{equation}
To observe its characteristic transition behavior, we set $\gamma = 0.1$, $V(x) = \frac{1}{4}(x^4 - 2x^2)$, $\sigma_t = 3$, and $H = 1/2$, focusing on the transition from the equilibrium state $(-1, 0)$ to $(1, 0)$. 

Rewriting the system into a first-order state-space form yields
\begin{equation*}
    \frac{d}{dt} \begin{bmatrix} X_t \\ Y_t \end{bmatrix} 
    = \begin{bmatrix} 
        Y_t \\ 
        -\gamma Y_t - V'(X_t)
    \end{bmatrix}.
\end{equation*}
The deterministic system possesses stable equilibria at $(\pm 1, 0)$, and noise injection induces transitions between these states. The corresponding OM functional for this system is given by
\begin{equation*}
    J(\psi) = \int_0^1 \frac{\bigl( \psi_t'' + \gamma \psi_t' + V'(\psi_t) \bigr)^2}{2\sigma^2} \, dt - \frac{\gamma}{2}.
\end{equation*} 
Exploiting the energy structure of the autonomous system and noting that the boundary terms $\left[ (\psi_t')^2 + 2V(\psi_t) \right]_0^1$ vanish for transitions between symmetric equilibria, the functional simplifies to:
\begin{align*}
    \tilde{J}(\psi) &= \int_0^1 \left[ \bigl( \psi_t'' + V'(\psi_t) \bigr)^2 + \gamma^2 (\psi_t')^2 \right] dt + \gamma \left[ (\psi_t')^2 + 2V(\psi_t) \right]_0^1 \\
    &= \int_0^1 \left[ \bigl( \psi_t'' + V'(\psi_t) \bigr)^2 + \gamma^2 (\psi_t')^2 \right] dt.
\end{align*}
Applying the calculus of variations, we derive the associated Euler--Lagrange equation characterizing the optimal path:
\begin{equation*}
    \psi_t^{(4)} + \bigl( 2V''(\psi_t) - \gamma^2 \bigr) \psi_t'' + V'''(\psi_t) \bigl( \psi_t' \bigr)^2 + V''(\psi_t) V'(\psi_t) = 0.
\end{equation*}
Boundary-value problem solutions were computed using the \texttt{bvp4c} algorithm. These theoretical trajectories align closely with empirical mean trajectories from stochastic simulations, validating our analytical framework (see Fig.~\ref{example2}).

Furthermore, we extend our numerical investigation to a non-autonomous regime by introducing time-dependent noise intensity, $\sigma_t = 2 + \sin(8\pi t)$, for fractional Hurst parameters $H \in \{0.3, 0.8\}$. In these cases, the interplay between fractional calculus and time-varying coefficients renders the explicit Euler--Lagrange equation analytically intractable. Consequently, the most probable paths are obtained by directly minimizing the discretized functional $J$ using an interior-point method. Figure~\ref{example3} illustrates the convergence of the average path toward the most probable path for $H=0.3$. For $H=0.8$, due to the numerical instability inherent in fractional differentiation, we restrict our focus to the empirical mean trajectory, as depicted in Figure~\ref{example4}.

These numerical experiments reveal a profound geometric insight: the most probable transition path is not only deformed by the memory effects governed by the Hurst parameter $H$, but also directly inherits morphological features—such as characteristic oscillations—from the time-modulated noise intensity $\sigma_t$.
\end{example}

    \begin{figure}[htbp]
    \centering
    % 宽度设置为 0.48\textwidth，完美适配单栏满宽或双栏中的单栏排版
    \includegraphics[width=0.48\textwidth]{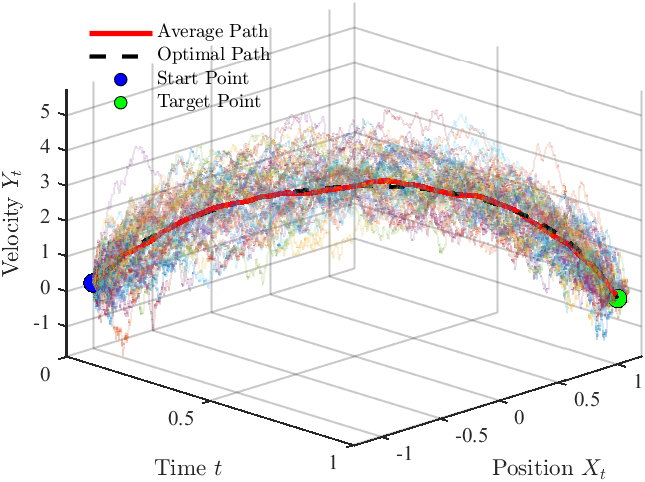}
    % 使用 \textit 保持数学变量与正文、图片内部符号的严格一致
    \caption{The average path and the optimal path for \eqref{equation 2}}
    \label{example2}
\end{figure}

\begin{figure}[htbp]
    \centering
    \includegraphics[width=0.48\textwidth]{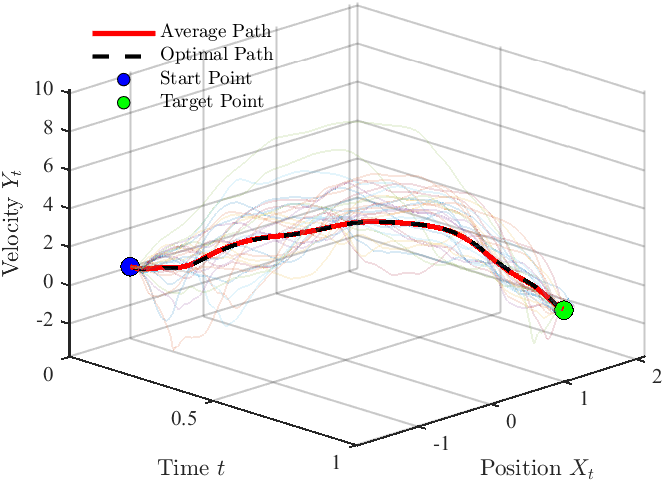}
    \caption{The average path and the optimal path for \eqref{equation 2}, with $H=0.3, \sigma_t = 2 +  \sin(8\pi t)$ and $\gamma=0.1$}
    \label{example3}
\end{figure}

\begin{figure}[htbp]
    \centering
    \includegraphics[width=0.48\textwidth]{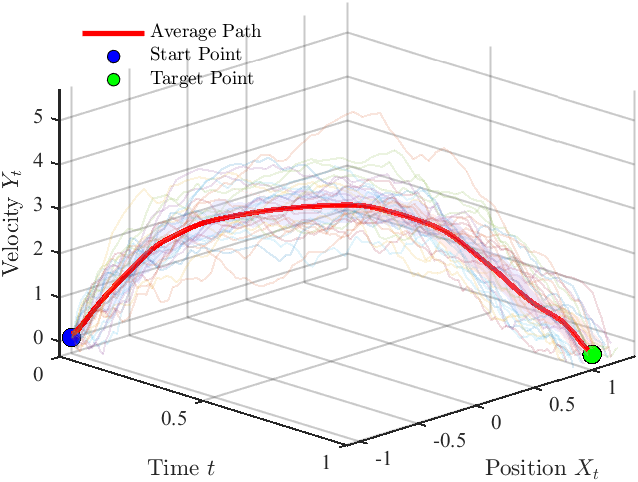}
    \caption{The average path  for \eqref{equation 2}, with $H=0.8, \sigma_t = 5 + 1.5\sin(4\pi t)$ and $\gamma=0.1$}
    \label{example4}
\end{figure}

\clearpage

\bibliographystyle{plain}
\bibliography{math.bib}

\end{document}